\documentclass{article}
\usepackage[T1]{fontenc}
\usepackage{lmodern}
 \usepackage{float}
\usepackage{amsmath,amssymb}
\usepackage{amssymb}
\usepackage{amsmath,amscd}
\usepackage{amsfonts}
\usepackage[utf8]{inputenc}
\usepackage{pstricks}
\usepackage{pstricks-add}
\usepackage{tikz}
\usepackage{url}
\usepackage{tensor}
\def\N{{\mathbb N}}

\usepackage{graphicx}
\usetikzlibrary{arrows}
\usepackage{MnSymbol}
\usepackage{stmaryrd}
\usepackage{shuffle}
\renewcommand{\S}{\ensuremath{\mathfrak{S}}}

\newcommand{\B}{\ensuremath{\mathcal{B}}}

\def\WSym{\mathrm{WSym}}
\def\BWSym{\mathrm{BWSym}}
\def\std{\mathrm{std}}

\newtheorem{theorem}{Theorem}

\newtheorem{claim}[theorem]{Claim}

\newtheorem{corollary}[theorem]{Corollary}

\newtheorem{definition}[theorem]{Definition}
\newtheorem{example}[theorem]{Example}

\newtheorem{lemma}[theorem]{Lemma}

\newtheorem{proposition}[theorem]{Proposition}
\newtheorem{remark}[theorem]{Remark}

\newenvironment{proof}[1][Proof]{\noindent\textbf{#1.} }{\ \rule{0.5em}{0.5em}}
\include {mak}
\PassOptionsToPackage{dvipsnames}{xcolor}
\RequirePackage{xcolor} 
\definecolor{halfgray}{gray}{0.55} 
\definecolor{webgreen}{rgb}{0,.5,0}
\definecolor{webbrown}{rgb}{.6,0,0}
\definecolor{Noir}{cmyk}{0, 0, 0, 0}
\definecolor{Bleu}{cmyk}{100, 72, 0, 18}
\definecolor{Rouge}{cmyk}{0, 100, 80, 2}

\tikzstyle{Arete}=[Rouge!80,cap=round,line width=3pt]
\tikzstyle{Feuille}=[rectangle,draw=Noir!70,fill=Noir!20,minimum size=3mm,line width=2pt]
\tikzstyle{Operateur} = [rectangle,rounded corners,draw=Bleu!100,fill=Bleu!10,
minimum size=15mm,line width=3pt,font=\Huge]

\newcounter{Edge}
\newcounter{Vertex}

\def\bugdxuu#1#2#3
{
\draw (#1,#2) rectangle (#1+0.5,#2+0.5);
\node[](m#3) at (#1+0.25,#2+0.25) {\arabic{Vertex}\stepcounter{Vertex}};
\node[] (dd1#3) at (#1+0.1,#2+0.1){};
\node[] (d1#3) at (#1+0.1,#2-0.3){\tiny\arabic{Edge}};

\node[] (uu1#3) at (#1+0.1,#2+0.5){};
\node[] (u1#3) at (#1+0.1,#2+0.7){\tiny\arabic{Edge}\stepcounter{Edge}};

\node[] (uu2#3) at (#1+0.4,#2+0.5){};
\node[] (u2#3) at (#1+0.4,#2+0.7){\tiny\arabic{Edge}};

\node[] (dd2#3) at (#1+0.4,#2+0.1){};
\node[] (d2#3) at (#1+0.4,#2-0.3){\tiny $\times\atop \arabic{Edge}$\stepcounter{Edge}};

\draw (dd1#3) -- (d1#3);
\draw (uu1#3) -- (u1#3);
\draw (uu2#3) -- (u2#3);
}

\def\bugdduu#1#2#3
{
\draw (#1,#2) rectangle (#1+0.5,#2+0.5);
\node[](m#3) at (#1+0.25,#2+0.25) {\arabic{Vertex}\stepcounter{Vertex}};
\node[] (dd1#3) at (#1+0.1,#2+0.1){};
\node[] (d1#3) at (#1+0.1,#2-0.3){\tiny\arabic{Edge}};

\node[] (uu1#3) at (#1+0.1,#2+0.5){};
\node[] (u1#3) at (#1+0.1,#2+0.7){\tiny\arabic{Edge}\stepcounter{Edge}};

\node[] (uu2#3) at (#1+0.4,#2+0.5){};
\node[] (u2#3) at (#1+0.4,#2+0.7){\tiny\arabic{Edge}};

\node[] (dd2#3) at (#1+0.4,#2+0.1){};
\node[] (d2#3) at (#1+0.4,#2-0.3){\tiny\arabic{Edge}\stepcounter{Edge}};

\draw (dd1#3) -- (d1#3);
\draw (uu1#3) -- (u1#3);
\draw (dd2#3) -- (d2#3);
\draw (uu2#3) -- (u2#3);
}

\def\bugddux#1#2#3
{
\draw (#1,#2) rectangle (#1+0.5,#2+0.5);
\node[](m#3) at (#1+0.25,#2+0.25) {\arabic{Vertex}\stepcounter{Vertex}};
\node[] (dd1#3) at (#1+0.1,#2+0.1){};
\node[] (d1#3) at (#1+0.1,#2-0.3){\tiny\arabic{Edge}};

\node[] (uu1#3) at (#1+0.1,#2+0.5){};
\node[] (u1#3) at (#1+0.1,#2+0.7){\tiny\arabic{Edge}\stepcounter{Edge}};

\node[] (uu2#3) at (#1+0.4,#2+0.5){};
\node[] (u2#3) at (#1+0.4,#2+0.7){\tiny$\times\atop \arabic{Edge}$\stepcounter{Edge}};

\node[] (dd2#3) at (#1+0.4,#2+0.1){};
\node[] (d2#3) at (#1+0.4,#2-0.3){\tiny\arabic{Edge}};

\draw (dd1#3) -- (d1#3);
\draw (uu1#3) -- (u1#3);
\draw (dd2#3) -- (d2#3);
}

\def\bugdxuu#1#2#3
{
\draw (#1,#2) rectangle (#1+0.5,#2+0.5);
\node[](m#3) at (#1+0.25,#2+0.25) {\arabic{Vertex}\stepcounter{Vertex}};
\node[] (dd1#3) at (#1+0.1,#2+0.1){};
\node[] (d1#3) at (#1+0.1,#2-0.3){\tiny\arabic{Edge}};

\node[] (uu1#3) at (#1+0.1,#2+0.5){};
\node[] (u1#3) at (#1+0.1,#2+0.7){\tiny\arabic{Edge}\stepcounter{Edge}};

\node[] (uu2#3) at (#1+0.4,#2+0.5){};
\node[] (u2#3) at (#1+0.4,#2+0.7){\tiny\arabic{Edge}};

\node[] (dd2#3) at (#1+0.4,#2+0.1){};
\node[] (d2#3) at (#1+0.4,#2-0.3){\tiny$\times\atop \arabic{Edge}$\stepcounter{Edge}};

\draw (dd1#3) -- (d1#3);
\draw (uu1#3) -- (u1#3);
\draw (uu2#3) -- (u2#3);
}

\def\bugdu#1#2#3
{
\draw (#1,#2) rectangle (#1+0.4,#2+0.5);
\node[](m#3) at (#1+0.2,#2+0.25) {\arabic{Vertex}\stepcounter{Vertex}};
\node[] (dd1#3) at (#1+0.2,#2+0.1){};
\node[] (d1#3) at (#1+0.2,#2-0.3){\tiny\arabic{Edge}};

\node[] (uu1#3) at (#1+0.2,#2+0.5){};
\node[] (u1#3) at (#1+0.2,#2+0.7){\tiny\arabic{Edge}\stepcounter{Edge}};

\draw (dd1#3) -- (d1#3);
\draw (uu1#3) -- (u1#3);
}

\def\bugddux#1#2#3
{
\draw (#1,#2) rectangle (#1+0.5,#2+0.5);
\node[](m#3) at (#1+0.25,#2+0.25) {\arabic{Vertex}\stepcounter{Vertex}};
\node[] (dd1#3) at (#1+0.1,#2+0.1){};
\node[] (d1#3) at (#1+0.1,#2-0.3){\tiny\arabic{Edge}};

\node[] (uu1#3) at (#1+0.1,#2+0.5){};
\node[] (u1#3) at (#1+0.1,#2+0.7){\tiny\arabic{Edge}\stepcounter{Edge}};

\node[] (uu2#3) at (#1+0.4,#2+0.5){};
\node[] (u2#3) at (#1+0.4,#2+0.7){\tiny$\arabic{Edge}\atop \times$};

\node[] (dd2#3) at (#1+0.4,#2+0.1){};
\node[] (d2#3) at (#1+0.4,#2-0.3){\tiny\arabic{Edge}\stepcounter{Edge}};

\draw (dd1#3) -- (d1#3);
\draw (uu1#3) -- (u1#3);
\draw (dd2#3) -- (d2#3);

}

\def\bugxdxu#1#2#3
{
\draw (#1,#2) rectangle (#1+0.5,#2+0.5);
\node[](m#3) at (#1+0.25,#2+0.25) {\arabic{Vertex}\stepcounter{Vertex}};
\node[] (dd1#3) at (#1+0.1,#2+0.1){};
\node[] (d1#3) at (#1+0.1,#2-0.3){\tiny$\times\atop \arabic{Edge}$};

\node[] (uu1#3) at (#1+0.1,#2+0.5){};
\node[] (u1#3) at (#1+0.1,#2+0.7){\tiny$\arabic{Edge}\atop\times$\stepcounter{Edge}};

\node[] (uu2#3) at (#1+0.4,#2+0.5){};
\node[] (u2#3) at (#1+0.4,#2+0.7){\tiny\arabic{Edge}};

\node[] (dd2#3) at (#1+0.4,#2+0.1){};
\node[] (d2#3) at (#1+0.4,#2-0.3){\tiny\arabic{Edge}\stepcounter{Edge}};

\draw (dd2#3) -- (d2#3);
\draw (uu2#3) -- (u2#3);

}

\def\bugdxxx#1#2#3
{
\draw (#1,#2) rectangle (#1+0.5,#2+0.5);
\node[](m#3) at (#1+0.25,#2+0.25) {\arabic{Vertex}\stepcounter{Vertex}};
\node[] (dd1#3) at (#1+0.1,#2+0.1){};
\node[] (d1#3) at (#1+0.1,#2-0.3){\tiny\arabic{Edge}};

\node[] (uu1#3) at (#1+0.1,#2+0.5){};
\node[] (u1#3) at (#1+0.1,#2+0.7){\tiny$\arabic{Edge}\atop\times$\stepcounter{Edge}};

\node[] (uu2#3) at (#1+0.4,#2+0.5){};
\node[] (u2#3) at (#1+0.4,#2+0.7){\tiny$\arabic{Edge}\atop\times$};

\node[] (dd2#3) at (#1+0.4,#2+0.1){};
\node[] (d2#3) at (#1+0.4,#2-0.3){\tiny$\times\atop\arabic{Edge}$\stepcounter{Edge}};

\draw (dd1#3) -- (d1#3);
}

\def\bugddxuux#1#2#3
{
\draw (#1,#2) rectangle (#1+0.7,#2+0.5);
\node[](m#3) at (#1+0.35,#2+0.25) {\arabic{Vertex}\stepcounter{Vertex}};
\node[] (dd1#3) at (#1+0.1,#2+0.1){};
\node[] (d1#3) at (#1+0.1,#2-0.3){\tiny\arabic{Edge}};

\node[] (uu1#3) at (#1+0.1,#2+0.5){};
\node[] (u1#3) at (#1+0.1,#2+0.7){\tiny\arabic{Edge}\stepcounter{Edge}};

\node[] (dd2#3) at (#1+0.35,#2+0.1){};
\node[] (d2#3) at (#1+0.35,#2-0.3){\tiny\arabic{Edge}};

\node[] (uu2#3) at (#1+0.35,#2+0.5){};
\node[] (u2#3) at (#1+0.35,#2+0.7){\tiny\arabic{Edge}\stepcounter{Edge}};

\node[] (uu3#3) at (#1+0.6,#2+0.5){};
\node[] (u3#3) at (#1+0.6,#2+0.7){\tiny$\arabic{Edge}\atop \times$};

\node[] (dd3#3) at (#1+0.6,#2+0.1){};
\node[] (d3#3) at (#1+0.6,#2-0.3){\tiny$\times\atop\arabic{Edge}$\stepcounter{Edge}};

\draw (dd1#3) -- (d1#3);
\draw (uu1#3) -- (u1#3);
\draw (dd2#3) -- (d2#3);
\draw (uu2#3) -- (u2#3);

}

\def\bugddduux#1#2#3
{
\draw (#1,#2) rectangle (#1+0.7,#2+0.5);
\node[](m#3) at (#1+0.35,#2+0.25) {\arabic{Vertex}\stepcounter{Vertex}};
\node[] (dd1#3) at (#1+0.1,#2+0.1){};
\node[] (d1#3) at (#1+0.1,#2-0.3){\tiny\arabic{Edge}};

\node[] (uu1#3) at (#1+0.1,#2+0.5){};
\node[] (u1#3) at (#1+0.1,#2+0.7){\tiny\arabic{Edge}\stepcounter{Edge}};

\node[] (dd2#3) at (#1+0.35,#2+0.1){};
\node[] (d2#3) at (#1+0.35,#2-0.3){\tiny\arabic{Edge}};

\node[] (uu2#3) at (#1+0.35,#2+0.5){};
\node[] (u2#3) at (#1+0.35,#2+0.7){\tiny\arabic{Edge}\stepcounter{Edge}};

\node[] (uu3#3) at (#1+0.6,#2+0.5){};
\node[] (u3#3) at (#1+0.6,#2+0.7){\tiny$\arabic{Edge}\atop \times$};

\node[] (dd3#3) at (#1+0.6,#2+0.1){};
\node[] (d3#3) at (#1+0.6,#2-0.3){\tiny\arabic{Edge}\stepcounter{Edge}};

\draw (dd1#3) -- (d1#3);
\draw (uu1#3) -- (u1#3);
\draw (dd2#3) -- (d2#3);
\draw (dd3#3) -- (d3#3);
\draw (uu2#3) -- (u2#3);

}

\def\bugddduuu#1#2#3
{
\draw (#1,#2) rectangle (#1+0.7,#2+0.5);
\node[](m#3) at (#1+0.35,#2+0.25) {\arabic{Vertex}\stepcounter{Vertex}};
\node[] (dd1#3) at (#1+0.1,#2+0.1){};
\node[] (d1#3) at (#1+0.1,#2-0.3){\tiny\arabic{Edge}};

\node[] (uu1#3) at (#1+0.1,#2+0.5){};
\node[] (u1#3) at (#1+0.1,#2+0.7){\tiny\arabic{Edge}\stepcounter{Edge}};

\node[] (dd2#3) at (#1+0.35,#2+0.1){};
\node[] (d2#3) at (#1+0.35,#2-0.3){\tiny\arabic{Edge}};

\node[] (uu2#3) at (#1+0.35,#2+0.5){};
\node[] (u2#3) at (#1+0.35,#2+0.7){\tiny\arabic{Edge}\stepcounter{Edge}};

\node[] (uu3#3) at (#1+0.6,#2+0.5){};
\node[] (u3#3) at (#1+0.6,#2+0.7){\tiny\arabic{Edge}};

\node[] (dd3#3) at (#1+0.6,#2+0.1){};
\node[] (d3#3) at (#1+0.6,#2-0.3){\tiny\arabic{Edge}\stepcounter{Edge}};

\draw (dd1#3) -- (d1#3);
\draw (uu1#3) -- (u1#3);
\draw (dd2#3) -- (d2#3);
\draw (dd3#3) -- (d3#3);
\draw (uu2#3) -- (u2#3);
\draw (uu3#3) -- (u3#3);

}

\def\bugxdxxuu #1#2#3
{
\draw (#1,#2) rectangle (#1+0.7,#2+0.5);
\node[](m#3) at (#1+0.35,#2+0.25) {\arabic{Vertex}\stepcounter{Vertex}};
\node[] (dd1#3) at (#1+0.1,#2+0.1){};
\node[] (d1#3) at (#1+0.1,#2-0.3){\tiny$\times\atop\arabic{Edge}$};

\node[] (uu1#3) at (#1+0.1,#2+0.5){};
\node[] (u1#3) at (#1+0.1,#2+0.7){\tiny$\arabic{Edge}\atop \times$\stepcounter{Edge}};

\node[] (dd2#3) at (#1+0.35,#2+0.1){};
\node[] (d2#3) at (#1+0.35,#2-0.3){\tiny\arabic{Edge}};

\node[] (uu2#3) at (#1+0.35,#2+0.5){};
\node[] (u2#3) at (#1+0.35,#2+0.7){\tiny\arabic{Edge}\stepcounter{Edge}};

\node[] (uu3#3) at (#1+0.6,#2+0.5){};
\node[] (u3#3) at (#1+0.6,#2+0.7){\tiny\arabic{Edge}};

\node[] (dd3#3) at (#1+0.6,#2+0.1){};
\node[] (d3#3) at (#1+0.6,#2-0.3){\tiny$\times\atop\arabic{Edge}$\stepcounter{Edge}};

\draw (dd2#3) -- (d2#3);
\draw (uu2#3) -- (u2#3);
\draw (uu3#3) -- (u3#3);

}

\def\bugxx#1#2#3
{
\draw (#1,#2) rectangle (#1+0.4,#2+0.5);
\node[](m#3) at (#1+0.2,#2+0.25) {\arabic{Vertex}\stepcounter{Vertex}};
\node[] (dd1#3) at (#1+0.2,#2+0.1){};
\node[] (d1#3) at (#1+0.2,#2-0.3){\tiny$\times\atop\arabic{Edge}$};

\node[] (uu1#3) at (#1+0.2,#2+0.5){};
\node[] (u1#3) at (#1+0.2,#2+0.7){\tiny$\arabic{Edge}\atop\times$\stepcounter{Edge}};

}

\def\bugxu#1#2#3
{
\draw (#1,#2) rectangle (#1+0.4,#2+0.5);
\node[](m#3) at (#1+0.2,#2+0.25) {\arabic{Vertex}\stepcounter{Vertex}};
\node[] (dd1#3) at (#1+0.2,#2+0.1){};
\node[] (d1#3) at (#1+0.2,#2-0.3){\tiny$\times\atop\arabic{Edge}$};

\node[] (uu1#3) at (#1+0.2,#2+0.5){};
\node[] (u1#3) at (#1+0.2,#2+0.7){\tiny\arabic{Edge}\stepcounter{Edge}};

\draw (uu1#3) -- (u1#3);
}

\def\bugdx#1#2#3
{
\draw (#1,#2) rectangle (#1+0.4,#2+0.5);
\node[](m#3) at (#1+0.2,#2+0.25) {\arabic{Vertex}\stepcounter{Vertex}};
\node[] (dd1#3) at (#1+0.2,#2+0.1){};
\node[] (d1#3) at (#1+0.2,#2-0.3){\tiny\arabic{Edge}};

\node[] (uu1#3) at (#1+0.2,#2+0.5){};
\node[] (u1#3) at (#1+0.2,#2+0.7){\tiny$\arabic{Edge}\atop\times$\stepcounter{Edge}};

\draw (dd1#3) -- (d1#3);
}

\def\mcomp#1#2{{\displaystyle\mathop\medstar^{#2}_{#1}}}
\def\comp#1#2{{\displaystyle\mathop\star^{#2}_{#1}}}
\def\aa{{\mathtt a}}
\def\raa{{\color{red}\mathtt a}}
\def\baa{{\color{blue}\mathtt a}}

\def\ee{\mathtt{e}}

\title{A combinatorial Hopf algebra for the boson normal ordering problem}
\author{Imad Eddine Bousbaa, Ali Chouria, and Jean-Gabriel Luque}
\begin{document}
\maketitle

\begin{abstract}
In the aim to understand the generalization of Stirling numbers occurring in the bosonic normal ordering problem, several combinatorial models have been proposed. In particular, Blasiak \emph{et al.} defined combinatorial objects allowing to interpret the number of $S_{\bf{r,s}}(k)$ appearing in the identity $(a^\dag)^{r_n}a^{s_n}\cdots(a^\dag)^{r_1}a^{s_1}=(a^\dag)^\alpha\displaystyle\sum S_{\bf{r,s}}(k)(a^\dag)^k a^k$, where $\alpha$ is assumed to be non-negative. These objects are used to define a combinatorial Hopf algebra which specializes to the enveloping algebra of the  Heisenberg Lie algebra. Here, we propose a new variant of this construction which admits a realization with variables. This means that we construct our algebra from a free algebra $\mathbb{C}\langle A \rangle$ using quotient and shifted product. The combinatorial objects (B-diagrams) are slightly different from those proposed by Blasiak \emph{et al.}, but give also a combinatorial interpretation of the generalized Stirling numbers together with a combinatorial Hopf algebra related to Heisenberg Lie algebra. The main difference comes from the fact that the B-diagrams have the same number of inputs and outputs. After studying the combinatorics and the enumeration of B-diagrams, we propose two constructions of algebras called Fusion algebra $\mathcal{F}$, defined using formal variable and another algebra $\mathcal{B}$ constructed directly from the B-diagrams. We show the connection between these two algebras and that $\mathcal{B}$ can be endowed with a Hopf structure. We recognize two already known combinatorial Hopf subalgebras of $\mathcal{B}$ : $\WSym$ the algebra of word symmetric functions indexed by set partitions and $\BWSym$ the algebra of biword symmetric functions indexed by set partitions into lists.
\end{abstract}

\section{Introduction}
In Quantum Field Theory, the concept of field allows the creation
and the annihilation of particles in any point of the space.
Like any quantum systems, a quantum field has an Hamiltonian
$H$ and the associated Hilbert space $\mathcal H$ is generated by the
eigenvectors of $H$. In the bra-cket notation, the Hilbert space is generated by the
vectors $|n\rangle$ assumed to be orthogonal ($\langle m| n\rangle=\delta_{n,m}$).
This representation is usually called Fock space; the vector $|n\rangle$ means that there are $n$ particles in the system.
The creation and annihilation operators, denoted respectively by $a^\dag$ and $a$ are non-Hermitian operators acting on the Fock space by
\begin{equation}\label{Fock}
a^\dag|n\rangle=\sqrt{n+1}|n+1\rangle\mbox{ and }a|n\rangle=\sqrt{n}|n-1\rangle.
\end{equation}
These operators generate the Heisenberg  algebra abstractly defined as the free algebra generated by the elements $a$, $a^\dag$ and quotiented by the relation
\begin{equation}\label{Hrel}
[a,a^\dag]=1.
\end{equation}
The normal ordering problem consists in computing $\langle z|F(a^\dag,a)|z\rangle$ where $F(a^\dag,a)$ is an operator of the Heisenberg algebra, $|z\rangle$ is an eigenstate of the annihilation operator $a$ ($a|z\rangle =z|z\rangle)$, and $\langle z|a^\dag=\langle z|z^*$. The strategy consists in sorting the letters on each terms of   $F(a^\dag,a)$ in such a way that all the letters $a^\dag$ are in the left and all the letters $a$  are in the right by using as many times as necessary the relation (\ref{Hrel}).

In a seminal paper, J. Katriel \cite{Kat} pioneered the study of the combinatorial aspects of the normal ordering problem. He established the normal-ordered expression of $(a^\dag a)^n$ in terms of Stirling numbers of second kind enumerating partitions of a set of $n$ elements into $k$ non empty subsets
\begin{equation}
(a^\dag a)^n=\sum_{i=1}^nS(n,i)(a^\dag)^i a^i.
\end{equation}
The investigation of the normal ordered expression of $((a^\dag)^r a^s)^n$ naturally gives rise to generalized Stirling numbers $S_{r,s}(n,k)$ and Bell polynomials \cite{BPS,BPS2}. The interpretations of some special cases are well known and related to combinatorial numbers \cite{CDH,BHPSD0}. For instance, for $r=2$ and $s=1$, the number $S_{2,1}(n,k)=\binom{n-1}{k-1}{n!\over k!}$ is the number partitions of $\{1,\dots,n\}$ into $k$ lists (also called Lah numbers).
More generally, Blasiak \emph{ et al.} \cite{MBP} studied the bosonic normal ordering problem, 
\begin{equation}
(a^\dag)^{r_n}a^{s_n}\cdots(a^\dag)^{r_1}a^{s_1}=(a^\dag)^{\alpha_n}\sum_{k\geq s_1} S_{\bf{r,s}}(k)(a^\dag)^k a^k,
\end{equation}
with $\mathbf{r}=(r_1,\dots,r_n)$, $\mathbf{s}=(s_1,\dots,s_n)$ and $\alpha_n=\sum_{i=1}^n(r_i-s_i)$.

They gave also a combinatorial interpretation of the sequence $S_{\bf{r,s}}(k)$ in terms of graphs-like combinatorial objects called \emph{bugs}. In a further work, Blasiak \emph{et al.} \cite{BPDSHP} constructed a combinatorial Hopf algebras based on bugs to explain the computations. The aim of our paper is to investigate a variant of their construction. Our version allows a realization with variables which projects on the Heisenberg-Weyl algebra and the connections with some other combinatorial Hopf algebras appear clearly. Readers interested by the topic of combinatorial interpretations of the normal ordering should also refer to \cite{PHDBlasiak,MMS,MS,MS15,Mansour}.

Our paper is organized as follow. In section \ref{sectioDiagram}, we introduce the combinatorial structure of $B$-diagrams and we give, in section \ref{algebra}, an algebraic structure based on these objects. Also, we describe a second algebra named \emph{Fusion algebra} which specializes to the Heisenberg-Weyl algebra. We show that the $B$-diagram algebra is isomorphic to a subalgebra of $\mathcal{F}$ which allows us to describe the normal ordering problem in terms of $B$-diagram. We study the Hopf structure of $B$-diagrams in section \ref{Hopf} and we identify, in section \ref{subalgebras}, two already known subalgebras : $\WSym$ and $\BWSym$.
\section{B-diagrams\label{sectioDiagram}}
\subsection{Definition and first examples}
Let us first introduce  notation. Let $\# E$ denote the cardinal of the set $E$, $\llbracket a,b\rrbracket:=\{a,a+1,\dots,b-1,b\}$ for any pairs of integers $a\leq b$, $E=E'\uplus E''$ when $E=E'\cup E''[n]$ and $E'\cap E''[n]=\emptyset$, where $E''[n]$ means that we add $n = \# E'$ to each integer occurring in $E''$. For example, if we set $E=\{1,2,3\}$, thus $E[2]=\{3,4,5\}$. 
\begin{definition}\label{DBDiag}
A \emph{B-diagram} is a $5$-tuple $G=(n,\lambda,E^\uparrow,E^\downarrow,E)$ such that
\begin{enumerate}
\item $n\in\N$,
\item $\lambda=[\lambda_1,\dots,\lambda_n]$ with $\lambda_i\in\N\setminus\{0\}$ for each $i$,
\item $E^\uparrow, E^\downarrow\subset \llbracket 1,\lambda_1+\cdots+\lambda_n\rrbracket$,
\item $E\subset \{(a,b):a\in E^\uparrow, b\in E^\downarrow, v(a)<v(b)\}$ where $v:\llbracket 1,\lambda_1+\cdots+\lambda_n\rrbracket \longrightarrow \llbracket1,n\rrbracket$
is defined by
$v(k)=i$ if $k\in \llbracket \lambda_1+\cdots+\lambda_{i-1}+1,\lambda_1+\cdots+\lambda_i\rrbracket$,
\item for each $a\in E^\uparrow$ and $b\in E^\downarrow$, the sets $\{(a,c):(a,c)\in E\}$ and $\{(c,b):(c,b)\in E\}$ contain at most one element.
\end{enumerate}
\end{definition}

Graphically, a B-diagram can be represented as a graph with $n$ vertices. The vertex $i$ has exactly $\lambda_i$ inner (resp. outer) half-edges labelled
 by $\llbracket \lambda_1+\cdots+\lambda_{i-1}+1,\lambda_1+\cdots+\lambda_i\rrbracket$. The inner (resp. outer) half edges which does not belong to $E^\downarrow$ (resp. $E^\uparrow$) are denoted by $\times$. An element of $E$ is represented by an edge relying an outer half edge $a$ of a vertex $i$ to an inner half edge $b$ of a vertex $j$ with $i<j$.

\begin{example}\rm
For instance, the B-diagram $G=(3,[3,1,2],\llbracket 1,5\rrbracket,\llbracket 1,6\rrbracket,\{(1,6), (2,4), (4,5)\})$
is represented in Figure \ref{bdiag1}
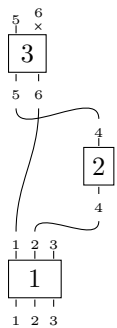
\begin{figure}[h]
\begin{center}
\begin{tikzpicture}
\setcounter{Edge}{1}
\setcounter{Vertex}{1}
\bugddduuu 00{b1}
\bugdu 1{1.5}{b2}
\bugddux 03{b3}
\draw (d1b2) edge[in=90,out=270]  (u2b1);
\draw (d1b3) edge[in=90,out=270] (u1b2);
\draw (d2b3) edge[in=90,out=270,] (u1b1);
\end{tikzpicture}
\end{center}
\caption{The B-diagram $(3,[3,1,2],\llbracket 1,5\rrbracket,\llbracket 1,6\rrbracket,\{(1,6), (2,4), (4,5)\})$\label{bdiag1}}
\end{figure}

Also consider   $G'=(4,[1,3,2,2],\{1,3,4,6\}, \{1,3,6,7\},\{(1,6),(3,7)\})$ which is represented in Figure \ref{bdiag2}.
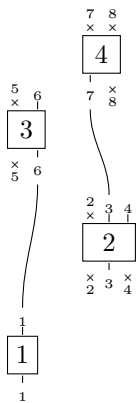
\begin{figure}[H]
\begin{center}
\begin{tikzpicture}
\setcounter{Edge}{1}
\setcounter{Vertex}{1}
\bugdu 00{b1}
\bugxdxxuu 1{1.5}{b2}
\bugxdxu 03{b3}
\bugdxxx 1{4}{b4}
\draw (d2b3) edge[in=90,out=270]  (u1b1);
\draw (d1b4) edge[in=90,out=270] (u2b2);
\end{tikzpicture}
\end{center}
\caption{The B-diagram $(4,[1,3,2,2],\{1,3,4,6\}, \{1,3,6,7\},\{(1,6),(3,7)\})$ \label{bdiag2}}
\end{figure}
\end{example}

Let $G=(n,\lambda,E^\uparrow,E^\downarrow,E)$, we define  some tools in order to manipulate more easily the B-diagrams
\begin{enumerate}
\item the number of vertices is $|G|:=n$,
\item we write $\omega(G):=\lambda_1+\cdots+\lambda_n$, for the number of half-edge,
\item Let $\tau(G):=\#E$, be the number of edges,
\item the number of non used outer (resp. inner) half edges is $h^{\uparrow}(G):=\omega(G)-\# e^{\uparrow}(E)$ (resp.
 $h^{\downarrow}(G):=\omega(G)-\# e^{\downarrow}(E)$) where $e^{\uparrow}(a,b)=a$ (resp. $e^{\downarrow}(a,b)=b$),
\item the set of non used outer (resp. inner) non cut half edges is $H_f^\uparrow(G):=E^\uparrow \setminus e^\uparrow(E)$
(resp. $H_f^\downarrow(G):=E^\downarrow \setminus e^\downarrow(E)$),
\item the number of non used outer (resp. inner) non cut half edges is $h_f^\uparrow(G):=\#H_f^\uparrow(G)$ (resp.
  $h_f^\downarrow(G):=\#H_f^\downarrow(G)$),
\item the number of non used cut half edges is $h_c(G):=h^{\uparrow}(G)-h^\uparrow_f(G)+h^{\downarrow}(G)-h^\downarrow_f(G)$,
\item the set of  half edges associated to a vertex $i$ is
 $\hat H(i):=\llbracket \lambda_1+\cdots+\lambda_{i-1}+1,\lambda_1+\cdots+\lambda_{i}\rrbracket$,
\item the set of outer (resp. inner) non cut half edges associated to a vertex $i$,
 $\hat H^\uparrow_f(i):=E^\uparrow\cap\llbracket \lambda_1+\cdots+\lambda_{i-1}+1,\lambda_1+\cdots+\lambda_{i}\rrbracket$
  (resp. $\hat H^\downarrow_f(i):=E^\downarrow\cap\llbracket \lambda_1+\cdots+\lambda_{i-1}+1,\lambda_1+\cdots+\lambda_{i}\rrbracket$),
\item we will also use  the map $v$ of Definition \ref{DBDiag}; this map will be denoted $v_G$ in case of ambiguity.
\end{enumerate}
\begin{example}
\rm
Consider the B-diagram $G=(n,\lambda,E^\uparrow,E^\downarrow,E)$ represented in Figure \ref{bdiag1}. We have
$|G|= 3$, $\omega(G)=6$, and $\tau(G)=3$. We also have  $h^\uparrow(G)=3$, $h^\uparrow_f(G)=2$, $h^\downarrow(G)=h^\downarrow_f(G)=3$, and $h_c(G)=1$ since
$H_f^\uparrow(G)=\{3,5\}$ and  $H_f^\downarrow(G)=\{1,2,3\}$. Furthermore $\hat H(1)=\hat H_f^\uparrow(1)=\hat H_f^\downarrow(1)=\{1,2,3\}$, $\hat H(2)=\hat H_f^\uparrow(2)=\hat H_f^\downarrow(2)=\{4\}$, $\hat H(3)=\hat H_f^\downarrow(3)=\{5,6\}$, and $\hat H_f^\uparrow(3)=\{5\}$. Finally, $v(1)=v(2)=v(3)=1$, $v(4)=2$, and $v(5)=v(6)=3$.
\end{example}
A special example B-diagram is given by the \emph{empty diagram} $\varepsilon:=(0,[],\emptyset,\emptyset,\emptyset)$. This is the only diagram of weight $0$.

\begin{definition}\label{subdiag}
Let $G=(n,\lambda,E^\uparrow,E^\downarrow,E)$ be a B-diagram. A \emph{sub B-diagram} of $G$ is completely characterized by a sequence $1\leq i_1<\cdots<i_{n'}\leq n$. More precisely, we define the B-diagram $G[i_1,\dots,i_{n'}]=(n',\lambda',E'^\uparrow,E'^\downarrow,E')$ by

\begin{enumerate}
\item $\lambda'=[\lambda_{i_1},\dots,\lambda_{i_{n'}}]$,
\item $E'^\uparrow=\phi\left(E^\uparrow\cap \bigcup_{\ell=1}^{n'}\hat H(i_\ell)\right)$ and $E'^\downarrow=\phi\left(E^\downarrow\cap \bigcup_{\ell=1}^{n'}\hat H(i_\ell)\right)$  where $\phi$ is the only increasing bijection sending $\bigcup_{\ell=1}^{n'}\hat H(i_\ell)$ to $\llbracket 1,\lambda_{i_1}+\cdots+\lambda_{i_{n'}}\rrbracket$.
\item $E'=\phi\left(E\cap \bigcup_{\ell=1}^{n'}\hat H^\uparrow_f(i_\ell)\times \hat H^\downarrow_f(i_\ell)\right)$
\end{enumerate}
Let $I=[i_1,\dots,i_{n'}]$ be a sequence of vertices of $G$, we let  $\complement_G I=\llbracket 1,n\rrbracket\setminus \{i_1,\dots,i_{n'}\}$ denote the \emph{complement} of $I$ in $G$.
\end{definition}

\begin{example}\rm
Let $G$ be the B-diagram of Figure \ref{bdiag1}. We have $\hat H(1)=\hat H^\uparrow_f(1)=\hat H^\downarrow_f(1)=\{1,2,3\}$, $\hat H(2)=\hat H^\uparrow_f(2)=\hat H^\downarrow_f(2)=\{4\}$, $\hat H(3)=\hat H^\downarrow_f(3)=\{5,6\}$, and $\hat H^\uparrow_f(3)=\{5\}$. Set $i_1=1$ and $i_2=3$. Following Definition \ref{subdiag}, we have
\begin{enumerate}
\item $\lambda'=[3,2]$,
\item $\phi$ sends respectively $1, 2, 3, 5, 6$ to $1, 2, 3, 4, 5$. Hence, $E'^\uparrow=\phi\left(\{1,2,3,5\}\right)=\{1,2,3,4\}$ and
$E'^\downarrow=\phi\left(\{1,2,3,5,6\}\right)=\{1,2,3,4,5\}$,
\item $E'=\phi\left(\{(1,6)\}\right)=\{(1,5)\}$.
\end{enumerate}
 We deduce that $G[1,3]=(2,[3,2],\llbracket 1,4\rrbracket,\llbracket 1,5 \rrbracket,\{(1,5)\})$ (see Figure \ref{Fsubdiag}).
\begin{figure}[h]
\begin{center}
\begin{tikzpicture}
\setcounter{Edge}{1}
\setcounter{Vertex}{1}
\bugddduuu 00{b1}
\bugddux 03{b3}
\draw (d2b3) edge[in=90,out=270,] (u1b1);
\end{tikzpicture}
\end{center}
\caption{The sub B-diagram $G[1,3]$ of $G=(3,[3,1,2],\llbracket 1,5\rrbracket,\llbracket 1,6\rrbracket,\{(1,6), (2,4), (4,5)\})$\label{Fsubdiag}}
\end{figure}
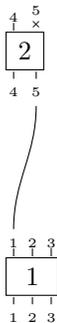

\end{example}
\subsection{Connections and compositions}
\begin{definition}
A B-diagram $G=(n,\lambda,E^\uparrow,E^\downarrow,E)$ is \emph{connected} if and only if for each $1\leq i<j\leq n$, there exists a sequence
of edges $(a_1,b_1),\dots,(a_k,b_k)\in E$ satisfying
\begin{enumerate}
\item $i\in\{v(a_1),v(b_1)\}$ and $j\in\{v(a_k),v(b_k)\}$,
\item for each $1\leq \ell<k$ one has $v(a_\ell)\in\{v(a_{\ell+1}),v(b_{\ell+1})\}$ or $v(b_\ell)\in\{v(a_{\ell+1}),v(b_{\ell+1})\}$.
\end{enumerate}
\end{definition}

A \emph{connected component} of a B-diagram $G$ is a sequence $i_1<\cdots<i_{n'}$ such that $G[i_1,\dots,i_{n'}]$ is a connected sub B-diagram which is maximal in the sense that if we add any vertex $i$ in the sequence $i_1<\cdots<i_{n'}$ then we obtain a sequence $i'_1<\cdots<i'_{n'+1}$ such that $G[i'_1,\dots,i'_{n'+1}]$ is not connected. Let $\mathrm{Connected}(G)$ denote the set of the connected components of $G$.
\begin{example}\rm
The B-diagram in Figure \ref{bdiag1} is connected whilst the B-diagram in Figure \ref{bdiag2} has two connected components $[1,3]$ and $[2,4]$.
\end{example}
A sequence $1\leq i_1<\dots,i_{n'}\leq n$ is \emph{isolated} in $G$  if for each $(a,b)\in E$ implies that $v(a)$ and $v(b)$ are both in
$\{i_1,\dots,i_{n'}\}$ or both in $\complement_G \{i_1,\dots,i_{n'}\}$.

\begin{claim}\label{isolation}The following assertions are equivalent
\begin{enumerate}
\item $I$ is isolated in $G$
\item $\complement_G I$ is isolated in $G$
\item There exist $j_1^1<\cdots<j^1_{k_1},\dots, j_1^p<\cdots<j_{k_p}^{p}$ such that $I=\{j_1^1,\cdots,j^1_{k_1},\dots, j_1^p,\cdots,j_{k_p}^{p}\}$ and each sequence $[j^\ell_1,\dots,j^\ell_{k_\ell}]$ is a connected component of $G$.
\end{enumerate}
\end{claim}

\begin{remark}
\begin{itemize}
\item A B-diagram $G$ and the empty diagram $\varepsilon$ are both isolated in $G$.
\item If $I$ is a connected component of $G$ then it is isolated in $G$.
\end{itemize}
\end{remark}
Let  $\mathrm{Iso}(G)$ denote the set isolated sequences in $G$ and set $\mathrm{Split}(G)=\left\{(I,\complement_GI):I\in\mathrm{Iso}(G)\right\}.$
\begin{definition}
Let $G=(n,\lambda,E^\uparrow,E^\downarrow,E)$ and $G'=(n',\lambda',E'^\uparrow,E'^\downarrow,E')$ be two B-diagrams. For any $k\geq 0$, any strictly increasing sequence $a_1<\dots<a_k$ in $H_f^\uparrow(G)$ and any $k$-tuple of distinct integers $b_1,\dots,b_k$ in $H_f^\downarrow(G')$, we define the \emph{composition} $\mcomp{a_1,\dots,a_k}{b_1,\dots,b_k}$ by
\[
\begin{array}{c}
G'\\
\mcomp{a_1,\dots,a_k}{b_1,\dots,b_k}\\
G
\end{array}=G'',
\]
where $G''$ is the $5$ tuple $(n+n',[\lambda_1,\dots,\lambda_n,\lambda'_1,\dots,\lambda'_n],E''^\uparrow,E''^\downarrow,E'')$ with
\begin{enumerate}
\item $E''^\uparrow=E'^\uparrow\cup\{i+\omega(G):i\in E'^\uparrow\}$ and $E''^\downarrow=E'^\downarrow\cup\{i+\omega(G):i\in E'^\downarrow\}$,
\item $E''=E \cup \{(a_\ell,b_\ell+\omega(G)):1\leq \ell\leq k\}\cup \{(a+\omega(G),b+\omega(G)):(a,b)\in E'\}$.
\end{enumerate}
\end{definition}
We easily check that the definition of composition is coherent with the structure of B-diagram
\begin{claim}
Let $G=(n,\lambda,E^\uparrow,E^\downarrow,E)$ and $G'=(n',\lambda',E'^\uparrow,E'^\downarrow,E')$ be two B-diagrams. For any $k\geq 0$, any strictly increasing sequence $a_1<\dots<a_k$ in $H_f^\uparrow(G)$ and any $k$-tuple of distinct integers $b_1,\dots,b_k$ in $H_f^\downarrow(G')$, the $5$-tuple $\begin{array}{c}
G'\\
\mcomp{a_1,\dots,a_k}{b_1,\dots,b_k}\\
G
\end{array}$ is a B-diagram.
\end{claim}
\begin{example}
\rm Figure \ref{comp1} gives an example of composition.
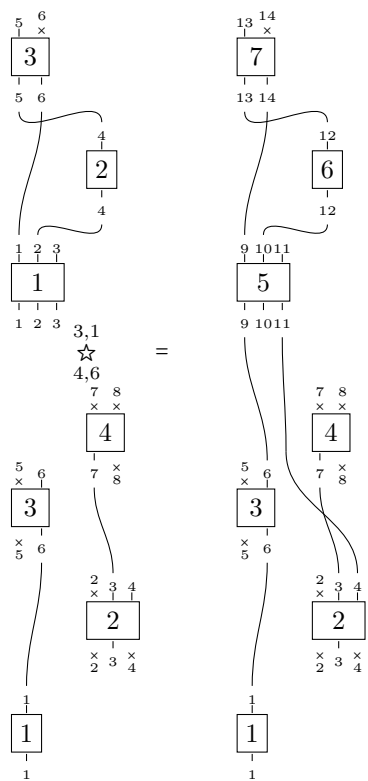
\begin{figure}[h]
\begin{center}
\begin{tikzpicture}
\setcounter{Edge}{1}
\setcounter{Vertex}{1}
\bugddduuu 06{b1}
\bugdu 1{7.5}{b2}
\bugddux 09{b3}
\draw (d1b2) edge[in=90,out=270]  (u2b1);
\draw (d1b3) edge[in=90,out=270] (u1b2);
\draw (d2b3) edge[in=90,out=270,] (u1b1);
\node (bla) at (1,5.3) {$\mcomp{4,6}{3,1} $};
\node (bla) at (2,5.3) {$=$};
\setcounter{Edge}{1}
\setcounter{Vertex}{1}
\bugdu 00{b1}
\bugxdxxuu 1{1.5}{b2}
\bugxdxu 03{b3}
\bugdxxx 1{4}{b4}
\draw (d2b3) edge[in=90,out=270]  (u1b1);
\draw (d1b4) edge[in=90,out=270] (u2b2);

\setcounter{Edge}{1}
\setcounter{Vertex}{1}

\bugdu 30{bb1}
\bugxdxxuu 4{1.5}{bb2}
\bugxdxu 33{bb3}
\bugdxxx 4{4}{bb4}
\draw (d2bb3) edge[in=90,out=270]  (u1bb1);
\draw (d1bb4) edge[in=90,out=270] (u2bb2);

\bugddduuu 36{bbb1}
\bugdu 4{7.5}{bbb2}
\bugddux 39{bbb3}
\draw (d1bbb2) edge[in=90,out=270]  (u2bbb1);
\draw (d1bbb3) edge[in=90,out=270] (u1bbb2);
\draw (d2bbb3) edge[in=90,out=270] (u1bbb1);

\draw (d3bbb1) edge[in=90,out=270] (3.65,4);
\draw (3.65,4) edge[in=90,out=270] (u3bb2);
\draw (d1bbb1) edge[in=90,out=270] (u2bb3);

\end{tikzpicture}
\end{center}
\caption{An example of composition\label{comp1}}
\end{figure}
\end{example}
Note that the special case where $k=0$ corresponds to a simple juxtaposition of the B-diagrams (see Figure \ref{comp2} for an example). We set $G'|G'':=\begin{array}{c} G''\\\medstar\\G'\end{array}$.
 The operation $|$ endows the set of the B-diagrams $\mathbb B$ with a structure of  monoid  whose unity is $\varepsilon$.
 \begin{definition}
 A B-diagram $G$ is \emph{indivisible} if $G=G'|G''$ implies $G'=G$ or $G''=G$. Let $\mathbb G$ denote the set of indivisible B-diagrams.
 \end{definition}
 It is easy to check that the indivisible diagrams are algebraically independent. It follows
 \begin{proposition}\label{BMfree}
 The monoid $(\mathbb B,|)$ is freely generated by $\mathbb G$: $\mathbb B \simeq\mathbb G^*$.
 \end{proposition}
\begin{figure}[h]
\begin{center}
\begin{tikzpicture}
\setcounter{Edge}{1}
\setcounter{Vertex}{1}
\bugdu 00{b1}
\bugxdxu 02{b3}
\draw (d2b3) edge[in=90,out=270]  (u1b1);

\setcounter{Edge}{1}
\setcounter{Vertex}{1}
\node(bla) at (0.3,3.2) {$\medstar$};
\node(eq) at (1,3.2) {$=$};
\bugxdxxuu 0{4}{b2}
\bugdxxx 0{6}{b4}
\draw (d1b4) edge[in=90,out=270] (u2b2);

\setcounter{Edge}{1}
\setcounter{Vertex}{1}
\bugdu 20{b1}
\bugxdxu 23{b3}
\bugxdxxuu 3{1.5}{b2}
\bugdxxx 3{4}{b4}
\draw (d2b3) edge[in=90,out=270]  (u1b1);
\draw (d1b4) edge[in=90,out=270] (u2b2);
\draw (d2b3) edge[in=90,out=270,] (u1b1);
\end{tikzpicture}
\end{center}
\caption{An example of composition when $k=0$ \label{comp2}}
\end{figure}
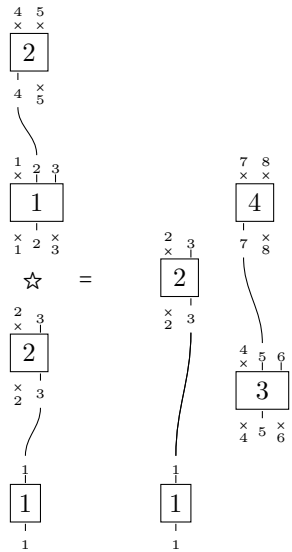
We can give an alternative recursive definition for the B-diagrams using the compositions.

\begin{lemma}\label{recBDiag}
Let $G=(n,\lambda,E^\uparrow,E^\downarrow,E)$ be a B-diagram. Either $G=\varepsilon$ or there exists a B-diagram
$V=(1,[p],E^\uparrow_v,E^\downarrow_v,E_v)$, a B-diagram $\tilde G=(n-1,\tilde \lambda,\tilde E^\uparrow,\tilde E^\downarrow,\tilde E)$ and two sequences $1\leq a_1<\dots<a_k\leq p$ and $1\leq b_1,\dots,b_k\leq \omega(\tilde G)$ distinct satisfying
\[
G=\begin{array}{c}
\tilde G\\
\mcomp{a_1,\dots,a_k}{b_1,\dots,b_k}\\
V
\end{array}.
\]
\end{lemma}
Let us set
\begin{equation}
G\medstar G':=\left\{\begin{array}{c}
G'\\
\mcomp{a_1,\dots,a_k}{b_1,\dots,b_k}\\
G
\end{array}: a_1<\dots<a_k\in H_f^\uparrow(G) b_1,\dots,b_k\in H_f^\downarrow(G')\mbox{ distinct and }k\geq 0\right\}.
\end{equation}
\subsection{Enumeration}
Let $d_{p,q}$ denote the number of B-diagram $G$ such that $\omega(G)=p$ and $h_f^\uparrow(G)=q$ (see in Figure \ref{dpq} the first values of $d_{p,q}$ or the sequence A265199 in \cite{Sloane}).

 From lemma \ref{recBDiag}, we find the induction
\begin{equation}\label{enumdiag}
d_{p,q}=\sum_{i=1}^p\sum_{j=0}^i\sum_{k=0}^i\sum_{\ell=0}^j \ell!\binom j\ell\binom {q-k+\ell}\ell\binom ij\binom ik d_{p-i,q-k+\ell},
\end{equation}
with the special cases $d_{0,0}=1$ and $d_{p,q}=0$ if $p,q\leq 0$ and $(p,q)\neq (0,0)$. Indeed, we obtain a diagram with $p$ half edges and $q$ non used outer non cut half edges by branching $\ell$ inner half edges of an elementary B-diagram with $i$ half edges, $j\leq i$ non used inner non cut half edges, and $k\leq i$ non used outer non cut half edges to a B-diagram with $p-i$ half edges and  $q-k+\ell$ non used outer non cut half edges. The number of ways to do that is $\ell!\binom j\ell\binom {q-k+\ell}\ell\binom ij\binom ik$. Indeed, the factor $\ell!$ is the number of permutation of the inner half edges of the elementary B-diagram, the factor $\binom j\ell$ corresponds to the choice of $\ell$ half edges in the set of the $j$ possible non cut half edges, the coefficients $\binom {q-k+\ell}\ell$ is the number of ways to choose $\ell$ outer half edges in the second B-diagram, and the factors $\binom ij\binom ik$ is the number of ways to select $j$ inner half edges and $k$ outer half edges in $i$ half edges.
\begin{figure}[h]
\[
\begin{array}{|c|c|c|c|c|c|c|c|}\hline
d_{p,q}&0&1&2&3&4&5&6\\\hline
0&1&&&&&&\\\hline
1&2&2&&&&&\\\hline
2&10&18&8&&&&\\\hline
3&62& 154& 124& 32&&&\\\hline
4&462& 1426& 1596& 760& 128&&\\\hline
5&3982& 14506& 20380& 13680& 4336& 512&\\\hline
6&38646& 161042& 269284& 229448& 104032& 23520& 2048\\\hline\end{array}
\]
\caption{First values of $d_{p,q}$.\label{dpq}}
\end{figure}
The number $\alpha_p$ of B-diagrams having exactly $p$ half edges is given by
\begin{equation}\label{countdiag}
\alpha_p=\sum_{q=0}^pd_{p,q}.
\end{equation}
The first values (see sequence A266093 in \cite{Sloane}) are
\[1, 4, 36, 372, 4372, 57396, 828020, 12962164, 218098356,\dots\]
\begin{example}\rm
Let us illustrate this enumeration by counting the B-diagrams of weight $3$. A B-diagram of weight $3$ is
\begin{enumerate}
\item either an elementary diagram of weight $3$ ($2^6$ possibilities),
\item or it is composed by an elementary diagram of weight $2$ and another elementary diagram of weight $1$. The two diagrams can be juxtaposed ($2^7$ possibilities) or branched $2^6$ possibilities.
\item or it is composed by three elementary diagrams of weight $1$. The three diagrams can be juxtaposed ($2^6$ possibilities), or two of the three diagrams are connected ($3\times 2^4$ possibilities), or the three diagrams are connected ($2^2$ possibilities).
\end{enumerate}
Hence, we obtain $2^6+2^7+2^6+2^6+3\times2^4+2^2=372$ as predicted by (\ref{countdiag}).
\end{example}

\section{Algebraic aspects of B-diagrams\label{algebra}}
\subsection{Fusion algebra}
We consider the alphabets \begin{itemize}\item $\mathcal A_{\langle\times}=\left\{\raa_{\left\langle\left(i_1\atop j_1\right)\cdots\left(i_k\atop j_k\right)\right.}:k\in\N, i_1,\dots,i_k,j_1,\dots,j_k\in\N\setminus\{0\}\right\}$,
\item $\mathcal A_{\rangle\times}=\left\{\raa_{\left\rangle\left(i_1\atop j_1\right)\cdots\left(i_k\atop j_k\right)\right.}:k\in\N, i_1,\dots,i_k,j_1,\dots,j_k\in\N\setminus\{0\}\right\}$,
\item $\mathcal A_{\times\langle}=\left\{\baa_{\left.\left(i_1\atop j_1\right)\cdots\left(i_k\atop j_k\right)\right\langle}:k\in\N, i_1,\dots,i_k,j_1,\dots,j_k\in\N\setminus\{0\}\right\}$,
\item and $\mathcal A_{\times\rangle}=\left\{\baa_{\left.\left(i_1\atop j_1\right)\cdots\left(i_k\atop j_k\right)\right\rangle}:k\in\N, i_1,\dots,i_k,j_1,\dots,j_k\in\N\setminus\{0\}\right\}$.
\end{itemize}
For simplicity, we also set  $\mathcal A_{\langle\ \rangle}=\mathcal A_{\langle\times}\cup\mathcal A_{\times\rangle}$,  $\mathcal A_{\rangle\ \langle}=\mathcal A_{\rangle\times}\cup\mathcal A_{\times\langle}$, $\mathcal A_{\bullet\times}=\mathcal A_{\langle\times}\cup\mathcal A_{\rangle\times}$, $\mathcal A_{\times\diamond}=\mathcal A_{\times\langle}\cup\mathcal A_{\times\rangle}$, and $\mathcal A=\mathcal A_{\langle\ \rangle}\cup\mathcal A_{\rangle\ \langle}$.
\begin{definition}\label{dfusion}
Let $\mathcal J_F$ be the ideal of $\mathbb C\langle \mathcal A\rangle$, the free associative algebra generated by $\mathcal A$, generated by the polynomials
\begin{enumerate}
\item \label{P1f}$[\mathtt e,\mathtt f]=\mathtt e\mathtt f-\mathtt f\mathtt e$ for any $\mathtt e\in \mathcal A_{\langle\ \rangle}$ and $\mathtt f\in\mathcal A$,
\item \label{P2f}$\mathtt e\mathtt f-\mathtt f\mathtt e$ for any $\mathtt e,\mathtt f\in\mathcal A_{\times\diamond}$ and any $\mathtt e,\mathtt f\in\mathcal A_{\bullet\times}$,
\item \label{P3f}$\mathtt u\raa_{\bullet\alpha}\mathtt v\baa_{\alpha \langle}\raa_{\rangle\beta}\mathtt x\baa_{\beta\diamond}\mathtt y-\mathtt u\raa_{\bullet\alpha}\mathtt v\raa_{\rangle\beta}\baa_{\alpha \langle}\mathtt x\baa_{\beta\diamond}\mathtt y-\mathtt u\raa_{\bullet\alpha\beta}\mathtt v\mathtt x\baa_{\alpha\beta\diamond}\mathtt y$ for any $\mathtt u,\mathtt v,\mathtt  x,\mathtt y\in\mathcal A^*$ (the free monoid generated by $\mathcal A$), $\alpha=\left(i_1\atop j_1\right)\cdots \left(i_k\atop j_k\right)$, $\beta=\left(i'_1\atop j'_1\right)\cdots \left(i'_{k'}\atop j'_{k'}\right)$ (with $k,k'\in\N$ and $i_1,\dots,i_k$, $j_1,\dots,j_k$, $i'_1,\dots,i'_{k'}$, $j'_1\dots,j'_{k'}\in\N\setminus \{0\}$), and $\bullet,\diamond\in\{\langle,\rangle\}$.
\end{enumerate}
We call \emph{fusion algebra} the quotient $\mathcal F:=\mathbb C\langle \mathcal A\rangle/_{\mathcal J_F}$.
\end{definition}
Remark that $\mathcal A_{\langle\ \rangle}$ is in the center of $\mathcal F$ and that the subalgebra of $\mathcal F$ generated by $\mathcal A_{\times\diamond}$ (resp. $\mathcal A_{\bullet\times}$) is commutative.\\
Using the rule \ref{P3f} of Definition \ref{dfusion}, we show  by induction the following result
\begin{lemma}
\begin{equation}\label{FP3}\begin{array}{l}
\raa_{\rangle\alpha_1}\dots \raa_{\rangle\alpha_n}\baa_{\alpha_1\langle}\dots \baa_{\alpha_n\langle}\raa_{\rangle\beta_1}\dots \raa_{\rangle\beta_m}\baa_{\beta_1\langle}\dots \baa_{\beta_m\langle}=\\\displaystyle \sum_{k=0}^{\min\{n,m\}}
\sum_{1\leq i_1<\dots<i_{k}\leq n\atop 1\leq j_1,\dots,j_{k}\leq m\mbox{ \tiny distinct}}P\left(i_1,\dots,i_{k}\atop j_1,\dots,j_{k}\right)\raa_{\rangle\alpha_{i_1}\beta_{i_1}}\dots
\raa_{\rangle\alpha_{i_{k}}\beta_{i_{k}}}
Q\left(i_1,\dots,i_{k}\atop j_1,\dots,j_{k}\right)
\baa_{\alpha_{i_1}\beta_{i_1}\langle}\dots
\baa_{\alpha_{i_{k}}\beta_{i_{k}}\langle}
\end{array},\end{equation}
where $P\left(i_1,\dots,i_{k}\atop j_1,\dots,j_{k}\right)=\prod_{j\not\in\{j_1,\dots,j_{k}\}}\raa_{\rangle\beta_j}\prod{i\not\in\{i_1,\dots,i_{k}\}}\raa_{\rangle\alpha_j}$ and\\ $Q\left(i_1,\dots,i_{k}\atop j_1,\dots,j_{k}\right)=\prod_{j\not\in\{j_1,\dots,j_{k}\}}\baa_{\beta_j\langle}\prod_{i\not\in\{i_1,\dots,i_{k}\}}\baa_{\alpha_j\langle}$.
\end{lemma}
\begin{proof}
We proceed by induction on $m$. If $m=0$ then the result is obvious. Otherwise applying successively many times the rule \ref{P3f} of Definition \ref{dfusion} and hence the rule \ref{P2f}, one obtains
\begin{equation}\begin{array}{l}\displaystyle
\raa_{\rangle\alpha_1}\dots \raa_{\rangle\alpha_n}\baa_{\alpha_1\langle}\dots \baa_{\alpha_n\langle}\raa_{\rangle\beta_1}\dots \raa_{\rangle\beta_m}\baa_{\beta_1\langle}\dots \baa_{\beta_m\langle}
=\raa_{\rangle\beta_1}
\prod_{i=1}^n\raa_{\rangle\alpha_i}
\prod_{i=1}^n\baa_{\alpha_i\langle}
\prod_{j=2}^n\raa_{\rangle\beta_j}
\prod_{j=2}^n\baa_{\beta_j\langle}\baa_{\beta_1\langle}\\
+\displaystyle\sum_{i=1}^n\raa_{\rangle\alpha_{i}\beta_1}
\prod_{j\neq i}\raa_{\rangle\alpha_j}
\prod_{j\neq i}\baa_{\alpha_j\langle}
\prod_{j=2}^m\raa_{\rangle\beta_j}
\prod_{j=2}^m\baa_{\beta_j\langle}\baa_{\alpha_i\beta_1\langle}.
\end{array}
\end{equation}
Hence, applying the induction hypothesis to  each term
\begin{equation}
\displaystyle  \prod_{j=1}^n\raa_{\rangle\alpha_i}\prod_{j=1}^n\baa_{\alpha_i\langle}\prod_{j=2}^n
\raa_{\rangle\beta_i}\prod_{j=2}^n\baa_{\beta_i\langle}\mbox{ and
}
\prod_{j\neq i}\raa_{\rangle\alpha_j}
\prod_{j\neq i}\baa_{\alpha_j\langle}\prod_{j=2}^m\raa_{\rangle\beta_j}\prod_{j=2}^m
\baa_{\beta_j\langle}\baa_{\alpha_i\beta_1\langle},\end{equation}
for $1\leq i\leq n$, and again the rule \ref{P2f} of Definition \ref{dfusion} for writing each factor in the suitable order, we find the result.
\end{proof}

Whilst Formula (\ref{FP3}) seems rather technical, it is easy to understand. Indeed, it means that each of the letters $\baa_{\alpha_1\langle},\dots, \baa_{\alpha_k\langle}$ can be paired or not with any of the letters $\raa_{\rangle\beta_1},\dots, \raa_{\rangle\beta_\ell}$.
\begin{example}\rm
Let us illustrate Formula (\ref{FP3}) on the following example
\[\begin{array}{rcl}
\raa_{\rangle\alpha_1}\raa_{\rangle\alpha_2}\baa_{\alpha_1\langle}\baa_{\alpha_2\langle}\raa_{\rangle\beta_1}\raa_{\rangle\beta_2}
\baa_{\beta_1\langle}\baa_{\beta_2\langle}&=&
\raa_{\rangle\beta_1}\raa_{\rangle\beta_2}\raa_{\rangle\alpha_1}\raa_{\rangle\alpha_2}\baa_{\beta_1\langle}\baa_{\beta_2\langle}\baa_{\alpha_1\langle}\baa_{\alpha_2\langle}\\
&&+\raa_{\rangle\beta_2}\raa_{\rangle\alpha_2}\raa_{\rangle\alpha_1\beta_1}\baa_{\beta_2\langle}\baa_{\alpha_2\langle}\baa_{\alpha_1\beta_1\langle}
+\raa_{\rangle\beta_2}\raa_{\rangle\alpha_1}\raa_{\rangle\alpha_2\beta_1} \baa_{\beta_2\langle}\baa_{\alpha_1\langle}\baa_{\alpha_2\beta_1\langle}
\\&&+\raa_{\rangle\beta_1}\raa_{\rangle\alpha_2}\raa_{\rangle\alpha_1\beta_2}  \baa_{\beta_1\langle}\baa_{\alpha_2\langle}\baa_{\alpha_1\beta_2\langle}
+\raa_{\rangle\beta_1}\raa_{\rangle\alpha_1}\raa_{\rangle\alpha_2\beta_2}  \baa_{\beta_1\langle}\baa_{\alpha_1\langle}\baa_{\alpha_2\beta_2\langle}
\\&&+\raa_{\rangle\alpha_1\beta_1}\raa_{\rangle\alpha_2\beta_2}\baa_{\alpha_1\beta_1\langle}\baa_{\alpha_2\beta_2\langle}+\raa_{\rangle\alpha_2\beta_1}\raa_{\rangle\alpha_1\beta_2}
\baa_{\alpha_2\beta_1\langle}\baa_{\alpha_1\beta_2\langle}
\end{array}.
\]
\end{example}

Now, we define $\widetilde{\mathcal F}$ as the subspace of $\mathcal F$ generated by the elements under the form $$\raa_{\bullet_1\alpha_1}\dots \raa_{\bullet_k\alpha_k}
\baa_{\alpha_1\diamond_1}\dots \baa_{\alpha_k\diamond_k},$$ with $\bullet_i,\diamond_i\in\{\langle, \rangle\}$, for any $1\leq i\leq k$ and $k\in\mathbb N$.
\begin{proposition}
 $\widetilde{\mathcal F}$ is a subalgebra of  ${\mathcal F}$.
\end{proposition}
\begin{proof}
It suffices to prove that is stable for the product. Let $\mathtt u=\raa_{\bullet_1\alpha_1}\dots \raa_{\bullet_k\alpha_k}
\baa_{\alpha_1\diamond_1}\dots \baa_{\alpha_k\diamond_k}$ and $\mathtt v=\raa_{\bullet'_1\alpha'_1}\dots \raa_{\bullet'_{k'}\alpha'_{k'}}
\baa_{\alpha'_1\diamond'_1}\dots \baa_{\alpha'_{k'}\diamond_{k'}}$. Let $i_1,\dots,i_\ell$ be the indices such that $\diamond_{i_1}=\dots=\diamond_{i_\ell}=\rangle$ and $\{j_1,\dots,j_{k-\ell}\}=\{1,\dots,k\}\setminus \{i_1,\dots,i_\ell\}$. In the same way, let $i'_1,\dots,i'_{\ell'}$ be the indices such that $\bullet_{i'_1}=\dots=\bullet_{i'_{\ell'}}=\langle$ and $\{j'_1,\dots,j'_{k'-\ell'}\}=\{1,\dots,k'\}\setminus \{i'_1,\dots,i'_{\ell'}\}$.
From the relations \ref{P1f} and \ref{P2f} of Definition  \ref{dfusion}, one has
\begin{equation}\mathtt u\mathtt v=\raa_{\bullet_{i_1}\alpha_{i_1}}\dots \raa_{\bullet_{i_{\ell}}\alpha_{i_\ell}}
\raa_{\langle\alpha_{i'_1}}\dots \raa_{\langle\alpha_{i'_{\ell'}}}\mathtt u'\mathtt v'
\baa_{\alpha_{i'_1}\diamond_{i'_1}}\dots \baa_{\alpha_{i'_{\ell'}}\diamond_{i'_{\ell'}}}
\baa_{\alpha_{i_1}\rangle}\dots \baa_{\alpha_{i_\ell}\rangle},
\end{equation}
with $\mathtt u'=\raa_{\bullet_{j_1}\alpha_{j_1}}\dots \raa_{\bullet_{j_{k-\ell}}\alpha_{i_{k-\ell}}} \baa_{\alpha_{j_1}\langle}\dots \baa_{\alpha_{j_{k-\ell}}\langle}$
and $\mathtt v'=\raa_{\rangle\alpha'_{j'_1}}\dots \raa_{\rangle\alpha_{j'_{k'-\ell'}}}
\baa_{\alpha'_{j'_1}\diamond_{j'_1}}\dots \baa_{\alpha_{j'_{k'-\ell'}}\diamond_{j'_{k'-\ell'}}}$.\\
Hence, one has only to prove that $\mathtt u'\mathtt  v'\in \widetilde{\mathcal F}$ which is a direct consequence of Equality (\ref{FP3}).
\end{proof}
\begin{example}\rm
We consider the element $\mathtt w_1=\raa_{\left\rangle\left(1\atop 1\right)\right.}\raa_{\left\rangle\left(1\atop 2\right)\right.}\baa_{\left(1\atop 1\right)\langle}\baa_{\left.\left(1\atop 2\right)\right\rangle},\ \mathtt w_2=
\raa_{\left\rangle\left(2\atop 1\right)\right.}\raa_{\left\rangle\left(2\atop 2\right)\right.}\baa_{\left.\left(2\atop 1\right)\right\langle}\baa_{\left.\left(2\atop 2\right)\right\rangle}\in\widetilde{\mathcal F}$ and $\mathtt w=\mathtt w_1\mathtt w_2$.
First from the point (\ref{P1f}) of Definition \ref{dfusion}, $\baa_{\left.\left(1\atop 2\right)\right\rangle}$ is in the center of the algebra $\mathcal F$. Hence,
\[
\mathtt w=\raa_{\left\rangle\left(1\atop 1\right)\right.}\raa_{\left\rangle\left(1\atop 2\right)\right.}
\baa_{\left(1\atop 1\right)\langle}
\raa_{\left\rangle\left(2\atop 1\right)\right.}
\raa_{\left\rangle\left(2\atop 2\right)\right.}\baa_{\left.\left(2\atop 1\right)\right\langle}{\baa_{\left.\left(1\atop 2\right)\right\rangle}}\baa_{\left.\left(2\atop 2\right)\right\rangle}.
\]
Now we use the point (\ref{P3f}) and compute
\[\begin{array}{rcl}
\mathtt w&=&\overbrace{\ }^{\mathtt u}\raa_{\left\rangle\left(1\atop 1\right)\right.}\overbrace{\raa_{\left\rangle\left(1\atop 2\right)\right.}}^{\mathtt v}
\baa_{\left(1\atop 1\right)\langle}
\raa_{\left\rangle\left(2\atop 1\right)\right.}
\overbrace{\raa_{\left\rangle\left(2\atop 2\right)\right.}}^{\mathtt x}\baa_{\left.\left(2\atop 1\right)\right\langle}
\overbrace{\baa_{\left.\left(1\atop 2\right)\right\rangle}
\baa_{\left.\left(2\atop 2\right)\right\rangle}}^{\mathtt y}\\
&=&\mathtt u\raa_{\left\rangle\left(1\atop 1\right)\right.}\mathtt v\raa_{\left\rangle\left(2\atop 1\right)\right.}\baa_{\left(1\atop 1\right)\langle}\mathtt x\baa_{\left.\left(2\atop 1\right)\right\langle}\mathtt y
+
\mathtt u\raa_{\left\rangle\left(1\atop 1\right)\left(2\atop 1\right)\right.}\mathtt v\mathtt x\baa_{\left.\left(1\atop 1\right)\left(2\atop 1\right)\right\rangle}\mathtt y
\\&=&
\raa_{\left\rangle\left(1\atop 1\right)\right.}\raa_{\left\rangle\left(1\atop 2\right)\right.}
\raa_{\left\rangle\left(2\atop 1\right)\right.}\baa_{\left(1\atop 1\right)\langle}
\raa_{\left\rangle\left(2\atop 2\right)\right.}\baa_{\left.\left(2\atop 1\right)\right\langle}\baa_{\left.\left(1\atop 2\right)\right\rangle}\baa_{\left.\left(2\atop 2\right)\right\rangle}+
\raa_{\left\rangle\left(1\atop 1\right)\left(2\atop 1\right)\right.}
\raa_{\left\rangle\left(1\atop 2\right)\right.}
\raa_{\left\rangle\left(2\atop 2\right)\right.}
\baa_{\left.\left(1\atop 1\right)\left(2\atop 1\right)\right\rangle}
\baa_{\left.\left(1\atop 2\right)\right\rangle}\baa_{\left.\left(2\atop 2\right)\right\rangle}.\end{array}
\]

We use again the point (\ref{P3f}) on the first terms of the last sum and find
\[\begin{array}{rcl}\mathtt w&=&
\raa_{\left\rangle\left(1\atop 1\right)\right.}\raa_{\left\rangle\left(1\atop 2\right)\right.}
\raa_{\left\rangle\left(2\atop 1\right)\right.}\raa_{\left\rangle\left(2\atop 2\right)\right.}\baa_{\left(1\atop 1\right)\langle}\baa_{\left.\left(2\atop 1\right)\right\langle}\baa_{\left.\left(1\atop 2\right)\right\rangle}\baa_{\left.\left(2\atop 2\right)\right\rangle}+
\raa_{\left\rangle\left(1\atop 1\right)\left(2\atop 2\right)\right.}\raa_{\left\rangle\left(1\atop 2\right)\right.}
\raa_{\left\rangle\left(2\atop 1\right)\right.}\baa_{\left.\left(2\atop 1\right)\right\langle}\baa_{\left.\left(1\atop 2\right)\right\langle}\baa_{\left.\left(1\atop1\right)\left(2\atop 2\right)\right\rangle}\\&&
+
\raa_{\left\rangle\left(1\atop 1\right)\left(2\atop 1\right)\right.}\raa_{\left\rangle\left(1\atop 2\right)\right.}
\raa_{\left\rangle\left(2\atop 2\right)\right.} \baa_{\left.\left(1\atop 1\right)\left(2\atop 1\right)\right\langle}\baa_{\left.\left(1\atop 2\right)\right\rangle}\baa_{\left.\left(2\atop 2\right)\right\rangle}.
\end{array}
\]
Finally, using the point (\ref{P2f}) of Definition \ref{dfusion}, we show $w\in\widetilde{\mathcal F}$. Indeed,
\[\begin{array}{rcl}\mathtt w&=&
\raa_{\left\rangle\left(1\atop 1\right)\right.}
\raa_{\left\rangle\left(1\atop 2\right)\right.}
\raa_{\left\rangle\left(2\atop 1\right)\right.}
\raa_{\left\rangle\left(2\atop 2\right)\right.}
\baa_{\left(1\atop 1\right)\langle}
\baa_{\left.\left(1\atop 2\right)\right\rangle}
\baa_{\left.\left(2\atop 1\right)\right\langle}
\baa_{\left.\left(2\atop 2\right)\right\rangle}+
\raa_{\left\rangle\left(1\atop 1\right)\left(2\atop 2\right)\right.}
\raa_{\left\rangle\left(1\atop 2\right)\right.}
\raa_{\left\rangle\left(2\atop 1\right)\right.}
\baa_{\left.\left(2\atop 1\right)\right\langle}
\baa_{\left.\left(1\atop1\right)\left(2\atop 2\right)\right\rangle}
\baa_{\left.\left(1\atop 2\right)\right\rangle}
\\&&
+
\raa_{\left\rangle\left(1\atop 1\right)\left(2\atop 1\right)\right.}
\raa_{\left\rangle\left(1\atop 2\right)\right.}
\raa_{\left\rangle\left(2\atop 2\right)\right.} 
\baa_{\left.\left(1\atop 1\right)\left(2\atop 1\right)\right\langle}
\baa_{\left.\left(1\atop 2\right)\right\rangle}
\baa_{\left.\left(2\atop 2\right)\right\rangle}.
\end{array}\]
\end{example}

To each element $\mathtt b\in\mathcal A$ we define $|\mathtt b|:=\max\{i_\ell:1\leq\ell\leq k\}$ and $\omega(n):=\max\{j_\ell:1\leq\ell\leq k\}$ for any $\mathtt b\in\mathcal A$ with $\mathtt b=\raa_{\bullet\left(i_1\atop j_1\right)\cdots \left(i_1\atop j_1\right)}$ or $\mathtt b=\baa_{\left(i_1\atop j_1\right)\cdots \left(i_1\atop j_1\right)\diamond}$ ($\bullet,\diamond\in\{\langle,\rangle\}$). Also we set, for any word $\mathtt w$ in $\mathcal A^*$, $|\mathtt w|=\omega(\mathtt w)=0$ if $\mathtt w$ is the empty word, $|\mathtt w|=\max\{|\mathtt u|,|\mathtt b|\}$ and $\omega(\mathtt w)=\max\{\omega(\mathtt u),\omega(\mathtt b)\}$ for $\mathtt w=\mathtt u\mathtt b$ with $\mathtt b\in\mathcal A$.\\
We define on $\mathbb C\langle \mathcal A\rangle$ the shifted product as the only associative product satisfying, for each $\mathtt u,\mathtt v\in\mathcal A^*$,  $\mathtt u\star\mathtt  v=\mathtt u\mathtt v[|\mathtt u|,\omega(\mathtt u)]$, where $\raa_{\bullet \left(i_1\atop j_1\right)\cdots \left(i_k\atop j_k\right)}[m,n]=\raa_{\bullet \left(i_1+n\atop j_1+m\right)\cdots \left(i_k+n\atop j_k+m\right)}$, $\baa_{\left(i_1\atop j_1\right)\cdots \left(i_k\atop j_k\right)\diamond}[m,n]=\baa_{\left(i_1+n\atop j_1+m\right)\cdots \left(i_k+n\atop j_k+m\right)\diamond}$ ($\bullet,\diamond\in\{\langle,\rangle\}$), and $\mathtt w[n,m]=\mathtt u[n,m]\mathtt b[n,m]$ if $\mathtt u\in\mathcal A^*$ and $\mathtt b\in\mathcal A$.
\begin{claim}
\begin{enumerate}
\item Let $P\in\mathcal J_F$. For any $Q\in\mathbb C\langle\mathcal A\rangle$ we have $P\star Q, Q\star P\in \mathcal J_F$. In consequence, the operation $\star$ is well defined on $\mathcal F$.
\item If $P, Q\in \widetilde{\mathcal F}$ then $P\star Q\in \widetilde{\mathcal F}$.
\end{enumerate}
\end{claim}
These properties implies that we can endow the space $\mathcal F$  with the associative product $\star$. Let  $\hat {\mathcal F}$ denote this algebra.
\begin{example}\label{ExFH}\rm
Consider the element $\mathtt w=
\raa_{\left\rangle\left(1\atop 1\right)\right.}
\raa_{\left\rangle\left(1\atop 2\right)\right.}
\baa_{\left(1\atop 1\right)\langle}
\baa_{\left.\left(1\atop 2\right)\right\rangle}
\star
\raa_{\left\rangle\left(1\atop 1\right)\right.}
\raa_{\left\rangle\left(1\atop 2\right)\right.}
\baa_{\left(1\atop 1\right)\langle}
\baa_{\left.\left(1\atop 2\right)\right\rangle}$.
One has
\[\begin{array}{rcl}
\mathtt w
&=&
\raa_{\left\rangle\left(1\atop 1\right)\right.}
\raa_{\left\rangle\left(1\atop 2\right)\right.}
\baa_{\left.\left(1\atop 1\right)\right\langle}
\baa_{\left.\left(1\atop 2\right)\right\rangle}
\raa_{\left\rangle\left(2\atop 3\right)\right.}
\raa_{\left\rangle\left(2\atop 4\right)\right.}
\baa_{\left.\left(2\atop 3\right)\right\langle}
\baa_{\left.\left(2\atop 4\right)\right\rangle}\\
&=&
\raa_{\left\rangle\left(1\atop 1\right)\right.}
\raa_{\left\rangle\left(1\atop 2\right)\right.}
\raa_{\left\rangle\left(2\atop 3\right)\right.}
\raa_{\left\rangle\left(2\atop 4\right)\right.}
\baa_{\left.\left(1\atop 1\right)\right\langle}
\baa_{\left.\left(1\atop 2\right)\right\rangle}
\baa_{\left.\left(2\atop 3\right)\right\langle}
\baa_{\left.\left(2\atop 4\right)\right\rangle}+
\raa_{\left\rangle\left(1\atop 1\right)\left(2\atop 4\right)\right.}
\raa_{\left\rangle\left(1\atop 2\right)\right.}
\raa_{\left\rangle\left(2\atop 3\right)\right.}
\baa_{\left.\left(1\atop1\right)\left(2\atop 4\right)\right\rangle}
\baa_{\left.\left(1\atop 2\right)\right\rangle}
\baa_{\left.\left(2\atop 3\right)\right\langle}\\&&
+
\raa_{\left\rangle\left(1\atop 1\right)\left(2\atop 3\right)\right.}
\raa_{\left\rangle\left(1\atop 2\right)\right.}
\raa_{\left\rangle\left(2\atop 4\right)\right.} 
\baa_{\left.\left(1\atop 1\right)\left(2\atop 3\right)\right\langle}
\baa_{\left.\left(1\atop 2\right)\right\rangle}
\baa_{\left.\left(2\atop 4\right)\right\rangle}.
\end{array}
\]
\end{example}

\subsection{The algebra of B-diagrams}
We consider the algebra $\B$ consisting in the space formally generated by the B-diagrams and endowed with the product $\star$ defined by
\begin{equation}\label{starHW}
G\star G'=\sum_{G''\in G\medstar G'}G''=\sum_{a_1<\dots<a_k\in H^\uparrow(G)\atop b_1,\dots,b_k\in H^\downarrow(G'),\mbox{ \tiny distinct}} \begin{array}{c}G'\\\mcomp{a_1,\dots,a_k}{b_1,\dots,b_k}\\G\end{array}.
\end{equation}
For simplicity, for any $1\leq a_1<\cdots<a_k\leq\omega(G)$ and $1\leq b_1,\dots,b_k\leq\omega(G')$ distinct, we set
\[
\begin{array}{c}G'\\\comp{a_1,\dots,a_k}{b_1,\dots,b_k}\\G\end{array}=\left\{\begin{array}{cl}
\begin{array}{c}G'\\\mcomp{a_1,\dots,a_k}{b_1,\dots,b_k}\\G\end{array}&\mbox{ if } a_1<\dots<a_k\in H^\uparrow(G) \mbox{ and } b_1,\dots,b_k\in H^\downarrow(G')\\0&\mbox{ otherwise}
\end{array}\right.
\]
With this notation we have
\begin{equation}\label{starHW2}
G\star G'=\sum_{1\leq a_1<\cdots<a_k\leq\omega(G)\atop 1\leq b_1,\dots,b_k\leq\omega(G'),\mbox{ \tiny distinct}} \begin{array}{c}G'\\\comp{a_1,\dots,a_k}{b_1,\dots,b_k}\\G\end{array}.
\end{equation}
We set $\mathcal B_k=\mathrm{span}\{G:\omega(G)=k\}$.  Note that $\mathcal B$ splits into the direct sum $\mathcal B=\bigoplus_k\mathcal B_k$ and the dimension of each space $\mathcal B_k$ is finite.\\
Straightforwardly from the definition, we check
\begin{claim}
$(\mathcal B,\star)$ is a graded algebra with finite dimensional graded component.
The unit of this algebra is the empty B-diagram $\varepsilon$.
\end{claim}
\begin{example}
\rm See figures \ref{star1} and \ref{prod2} for two examples of product.
\end{example}
We remark that the product is triangular in the sense that all the B-diagrams different from $G''=G|G'$ which appear in the product $G\star G'$ have strictly less connected components than $G''$. Hence, since $\mathcal B= \mathbb C[\mathbb B]$  and $\mathbb B$ is isomorphic to the free monoid generated by $\mathbb G^*$ (see Proposition(\ref{BMfree})), we obtain the following result.
\begin{proposition}\label{BisFree}
The algebra $\mathcal B$ is free on the indivisible B-diagrams. In other words, it is isomprphic to $\B := \mathbb C\langle \mathbb G\rangle$.
\end{proposition}

\subsection{From B-diagrams algebra to Fusion algebra}
The aim of this section is to prove that the algebra $\mathcal{B}$ is isomorphic to a subalgebra of $\hat{\mathcal F}$.
\begin{definition}
A \emph{path} in a B-Diagram $G=(n,\lambda,E^\uparrow,E^\downarrow,E)$ is an increasing sequence of integers $1\leq i_1<\cdots<i_k\leq \omega(G)$ such that $(i_1,i_2),(i_2,i_3),\dots, (i_{k-1},i_k)\in E$, $i_1\in H_f^\downarrow(G)$ and $i_k\in H_f^\uparrow(G)$. Let $\mathrm{Paths}(G)$ denote the set of the paths in $G$.

Let $p=(i_1,\dots,i_k)\in\mathrm{Path}(G)$. For simplicity, we define
\begin{itemize}
\item $p^\downarrow=i_1$ and $p^\uparrow=i_k$,
\item $\bullet_p:=\left\{\begin{array}{ll}\ \rangle&\mbox{ if }i_1\in\mathcal H_f^\downarrow\\\ \langle&\mbox{otherwise} \end{array}\right.$ and   $\diamond_p=\left\{\begin{array}{ll}\ \langle&\mbox{ if }i_k\in\mathcal H_f^\uparrow\\\ \rangle&\mbox{otherwise} \end{array}\right.$
\item $\mathrm{seq}_G(p)=\left(v_G(i_1)\atop i_1\right)\cdots\left(v_G(i_k)\atop i_k\right)$
\end{itemize}
\end{definition}
\begin{remark}\label{rp2h}
It is easy to check that
\begin{enumerate}
\item $h_f^\uparrow=\#\{p\in\mathrm{Paths}(G):\diamond_p=\langle\}$,
\item $h_f^\downarrow=\#\{p\in\mathrm{Paths}(G):\bullet_p=\rangle\}$,
\item $h_c=\#\{p\in\mathrm{Paths}(G):\diamond_p=\rangle\}+\#\{p\in\mathrm{Paths}(G):\bullet_p=\langle\}$.
\end{enumerate}
\end{remark}
\begin{example}\rm
Consider the B-diagram $G$ of Figure \ref{bdiag1}. One has
\[
\mathrm{Paths}(G)=\{(1,6),(2,4,5),(3)\}.
\]
Suppose that $p=(1,6)$, one has
\[
\bullet_p=\rangle,\  \diamond_p=\rangle,\mbox{ and }\mathrm{seq}_G(p)=\left(1\atop 1\right)\left(3\atop 6\right).
\]
We have $h_f^\uparrow=2=\#\{(2,4,5),(3)\}$, $h_f^\downarrow=3=\#\{(1,6),(2,4,5),(3)\}$, and $h_c=1=\#\{(1,6)\}$.

\end{example}
For any B-diagram $G$ we define \begin{equation}\mathtt w(G):=\displaystyle\prod_{p\in \mathrm{Paths}(G)}\raa_{\bullet_p\mathrm{seq}_G(p)}\prod_{p\in \mathrm{Paths}(G)}\baa_{\mathrm{seq}(p)_G\diamond_p}\in\hat{\mathcal F}.\end{equation}
Clearly, a B-diagram $G$ is completely characterized by the values of $\mathtt{seq}(p)$, $\bullet_p$ and $\diamond_p$ associated to each of its paths $p$. So $\mathtt w$ is into. Furthermore, $\mathtt w$ sends the empty B-diagram $\varepsilon$ to the unit $1$ of $\hat{\mathcal F}$.
\begin{example}\rm
If we consider the B-diagram $G=(1,6)$ in Figure \ref{bdiag1}, one has
\[
\mathtt w(G)=\raa_{\left\rangle\left(1\atop 1\right)\left(3\atop 6\right)\right.}\raa_{\left\rangle\left(1\atop 2\right)\left(2\atop 4\right)\left(3\atop 5\right)\right.}
\raa_{\left\rangle\left(1\atop 3\right)\right.}\baa_{\left.\left(1\atop 1\right)\left(3\atop 6\right)\right\rangle}\baa_{\left.\left(1\atop 2\right)\left(2\atop 4\right)\left(3\atop 5\right)\right\rangle}
\baa_{\left.\left(1\atop 3\right)\right\rangle}.
\]
\end{example}
\begin{definition}
Let  $\mathcal F_{\mathcal H}$ denote the subalgebra of $\hat{\mathcal F}$ generated by the elements $\mathtt w(G)$ where $G$ is a B-diagram.
\end{definition}
\begin{lemma}\label{LbasisFH}
The set $\{\mathtt w(G):G\mbox{ is a B-diagram\}}$ is a basis of the space $\mathcal F_{\mathcal H}$.
\end{lemma}
\begin{proof}
One has to prove that for any $G_1,\dots, G_n$ distinct, $\sum \alpha_i\mathtt w(G_i)=0$ implies $\alpha_1=\cdots=\alpha_n=0$.
We proceed as follows: first we consider the partially commutative free algebra $\mathbb C\langle A,\theta\rangle=\mathbb C\langle A\rangle/_{\mathcal J_1}$ where $\mathcal J_1$ is the ideal generated by the polynomials of the points \ref{P1f} and \ref{P2f} of Definition \ref{dfusion}. Notice that $\mathbb C\langle A,\theta\rangle=\mathbb C[\mathbb M(\mathcal A,\theta)]$ is the algebra of the free partially commutative monoid $\mathbb M(\mathcal A,\theta)=\mathcal A^*/_{\equiv_\theta}$ where $\equiv_\theta$ is the congruence generated by $\mathtt u\mathtt e\mathtt f\mathtt v=\mathtt u\mathtt f\mathtt e\mathtt v$ for each $\mathtt u, \mathtt v\in\mathcal A^*$, ($\mathtt e\in\mathcal A_{\langle\ \rangle}$ and $\mathtt f\in \mathcal A$) or $\mathtt e,\mathtt f\in\mathcal A_{|\diamond}$ or $\mathtt e,\mathtt f\in\mathcal A_{\bullet|}$. Hence, we define $\mathcal J_2$ the ideal of $\mathbb C\langle A,\theta\rangle$ generated by the polynomials \begin{equation}\label{leqpol}\mathtt u\raa_{\bullet\alpha}\mathtt v\baa_{\alpha \langle}\raa_{\rangle\beta}\mathtt x\baa_{\beta\diamond}\mathtt y-\mathtt u\raa_{\bullet\alpha}\mathtt v\raa_{\rangle\beta}\baa_{\alpha \langle}\mathtt x\baa_{\beta\diamond}\mathtt y-\mathtt u\raa_{\bullet\alpha\beta}\mathtt v\mathtt x\baa_{\alpha\beta\diamond}\mathtt y\end{equation} for any $\mathtt u,\mathtt v, \mathtt x,\mathtt y\in\mathbb M(\mathcal A,\theta)$,  $\alpha=\left(i_1\atop j_1\right)\cdots \left(i_k\atop j_k\right)$, $\beta=\left(i'_1\atop j'_1\right)\cdots \left(i'_{k'}\atop j'_{k'}\right)$ (with $k,k'\in\N$ and $i_1,\dots,i_k$, $j_1,\dots,j_k$, $i'_1,\dots,i'_{k'}$, $j'_1\dots,j'_{k'}\in\N\setminus \{0\}$), and $\bullet,\diamond\in\{\langle,\rangle\}$. Observe that $\mathcal F=\mathbb C[\mathbb M(\mathcal A,\theta)]/_{\mathcal J_2}$. We define $\widetilde{\mathtt w}(G)=\displaystyle\prod_{p\in \mathrm{Paths}(G)}\raa_{\bullet_p\mathrm{seq}_G(p)}\prod_{p\in \mathrm{Paths}(G)}\baa_{\mathrm{seq}(p)_G\diamond_p}\in \mathbb C\langle A,\theta\rangle$. Let $\mathcal W$ be the subspace generated by the monomials $\widetilde{\mathtt w}(G)$.
Remarking that the map $\widetilde{\mathtt w}$ is into, our statement is equivalent to $\mathcal W\cap\mathcal J_2$=0.
Let $P\in \mathcal W\cap\mathcal J_2$. We have
$$\begin{array}{rcl}P&=&Q\left(\mathtt u\raa_{\bullet\alpha}\mathtt v\baa_{\alpha \langle}\raa_{\rangle\beta}\mathtt x\baa_{\beta\diamond}\mathtt y-\mathtt u\raa_{\bullet\alpha}\mathtt v\raa_{\rangle\beta}\baa_{\alpha \langle}\mathtt x\baa_{\beta\diamond}\mathtt y-\mathtt u\raa_{\bullet\alpha\beta}\mathtt v\mathtt x\baa_{\alpha\beta\diamond}\mathtt y\right)R\\
&=&\displaystyle\sum_i\alpha_i\prod_{p\in \mathrm{Paths}(G_i)}\raa_{\bullet_p\mathrm{seq}_{G_i}(p)}\prod_{p\in \mathrm{Paths}(G_i)}\baa_{\mathrm{seq}(p)_{G_i}\diamond_p}.\end{array}$$
Since $\mathtt u\raa_{\bullet\alpha}\mathtt v\baa_{\alpha \langle}\raa_{\rangle\beta}\mathtt x\baa_{\beta\diamond}\mathtt y$ can not be written under the form
$\prod_{j}\raa_{\bullet_j\alpha_j}\prod_{j}\baa_{\alpha_j\diamond_j}$, this is not possible unless $P=0$ (because of the factor $\baa_{\alpha \langle}\raa_{\rangle\beta}$ ).
\end{proof}

The behaviour of the paths with respect to the composition is summarized as follows:
\begin{equation}\label{comp2Path}\begin{array}{l}\mathrm{Paths}\left(\begin{array}{c}G'\\\mcomp{a_1,\dots,a_k}{b_1,\dots,b_k}\\G\end{array}\right)=\left\{p\in \mathrm{Paths}(G):p^\uparrow\not\in\{a_1,\dots,a_k\}\right\}
 \cup  \left\{p'[\omega(G)]:p'\in \mathrm{Paths}(G'):p^\downarrow\not\in\{b_1,\dots,b_k\}\right\}
\\ \cup \left\{pp'[\omega(G)]:p\in \mathrm{Paths}(G),
p'\in \mathrm{Paths}(G'), p^\uparrow\in\{a_1,\dots,a_k\}, p'^\downarrow\in\{b_1,\dots,b_k\}\right\}
\end{array}
\end{equation}
where $pp'$ denotes the sequence obtained by catening $p$ and $p'$ and $p[n]$ is the sequence obtained from $p$ by adding $n$ to each element.

\begin{example}\rm
Examine Figure \ref{comp1}. Let $G$ denote the lower B-diagram in the left hand sides of the equality and by $G'$ . We have
\[
\mathrm{Paths}(G)=\{(1,6),(2),(3,7),(4),(5),(8)\} \mbox{ and } \mathrm{Paths}(G')=\{(1,6),(2,4,5),(3)\}
\]
Hence
\[
\mathrm{Paths}\left(\begin{array}{c}G'\\\mcomp{4,6}{3,1} \\G\end{array}\right)=\{(2),(3,7),(5),(8) \}\cup\{
(10,12,13)\}\cup\{(1,6,9,14),(4,11)\}
\}
\]
\end{example}
\begin{theorem}
The algebras $\mathcal {B}$ and $\mathcal F_{\mathcal H}$ are isomorphic and an explicit isomorphism sends each B-diagram $G$ to $\mathtt w(G)$.
\end{theorem}
\begin{proof}
First let us prove that $\mathtt w$ can be extended as a morphism of algebra.
In other words, we first extend $\mathtt w$ as linear map and we prove that $\mathtt w\left(G\star G'\right)=\mathtt w(G)\star \mathtt w(G')$.
Observe that from (\ref{starHW}), we obtain
\begin{equation}\label{wG*G'}
\mathtt w\left(G\star G'\right) =\sum_{i_1<\dots<i_k\in H_f^\uparrow(G)\atop j_1,\dots,j_k\in H_f^\downarrow(G'),\mbox{ \tiny distinct}}\mathtt w\left( \begin{array}{c}G'\\\mcomp{i_1,\dots,i_k}{j_1,\dots,j_k}\\G\end{array}\right).
\end{equation}
But if $G''=\begin{array}{c}G'\\\mcomp{i_1,\dots,i_k}{j_1,\dots,j_k}\\G\end{array}$, Equality (\ref{comp2Path}) implies
\begin{equation}\label{wG''}\begin{array}{rcl}
\mathtt w\left(G'' \right)&=&\displaystyle
\prod^{(*)} \raa_{\bullet_p\mathtt{seq}_{G''}(p)} \prod^{(**)} \raa_{\bullet_{p'}\mathtt{seq}_{G''}\left(p'\right)}
\prod^{(***)}\raa_{\bullet_p\mathtt{seq}_{G''}\left(pp'\right)}\\&&\times\displaystyle
\prod^{(*)} \baa_{\mathtt{seq}_{G''}(p)\diamond_p} \prod^{(**)} \baa_{\mathtt{seq}_{G''}\left(p'\right)\diamond_{p'}}
\prod^{(***)}\baa_{\mathtt{seq}_{G''}\left(pp'\right)\diamond_{p'}}\end{array}\end{equation}
where the products $\displaystyle\prod^{(*)}$ are over the paths $p\in \mathrm{Paths}(G)$ such that $p^\uparrow\not\in\{i_1,\dots,i_k\}$,
the products $\displaystyle\prod^{(**)}$ are over the paths $p'\in \mathrm{Paths}(G')$ such that $p'^\downarrow\not\in\{j_1,\dots,j_k\}$, and
the products $\displaystyle\prod^{(***)}$ are over the pairs of paths $p\in \mathrm{Paths}(G)$ and $p'\in \mathrm{Paths}(G')$
with $p^\uparrow=i_h$ and $p'^\downarrow=j_h$ for some $1\leq h\leq k$.

Now, we examine $\mathtt w(G)\star \mathtt w(G')$. One has
\begin{eqnarray*}\label{wG*wG'}
\mathtt{w}(G)\star \mathtt{w}(G')&=&\displaystyle\left(\prod_{p\in\mathrm{Paths}(G)}\raa_{\bullet_p\mathrm{seq}_G(p)}\prod_{p\in\mathrm{Paths}(G)}\baa_{\mathrm{seq}_G(p)\diamond_p}\right)
\star \left(\prod_{p'\in\mathrm{Paths}(G')}\raa_{\bullet_{p'}\mathrm{seq}_{G'}(p')}\prod_{p'\in\mathrm{Paths}(G')}\baa_{\mathrm{seq}_{G'}(p')\diamond_{p'}}\right) \\ 
&=&\displaystyle\prod_{p\in\mathrm{Paths}(G)\atop \diamond_p=\rangle}\raa_{\bullet_p\mathrm{seq}_{G}(p)}
\prod_{p'\in\mathrm{Paths}(G')\atop \bullet_{p'}=\langle}\left(\raa_{\langle\mathrm{seq}_{G'}\left(p\right)}[|G|,\omega(G)]\right)\left(\mathtt u\star\mathtt  u'\right)\\ 
&& \times\displaystyle\prod_{p\in\mathrm{Paths}(G)\atop \diamond_p=\rangle}\baa_{\mathrm{seq}_{G}(p)\rangle}
\prod_{p'\in\mathrm{Paths}(G')\atop \bullet_{p'}=\langle}\left(\baa_{\mathrm{seq}_{G'}\left(p'\right)\diamond_{p'}}[|G|,\omega(G)]\right)
\end{eqnarray*}
where
\begin{equation} \mathtt u=\displaystyle\prod_{p\in\mathrm{Paths}(G)\atop \diamond_p=\langle}\raa_{\bullet_p\mathrm{seq}_G(p)}\prod_{p\in\mathrm{Paths}(G)\atop \diamond_p=\langle}\baa_{\mathrm{seq}_G(p)\langle}\mbox{ and }\mathtt u'=\displaystyle\prod_{p'\in\mathrm{Paths}(G')\atop \bullet_{p'}=\rangle}\raa_{\rangle\mathrm{seq}_{G'}(p')}\prod_{p'\in\mathrm{Paths}(G')\atop \bullet_{p'}=\rangle}\baa_{\mathrm{seq}_{G'}(p')\diamond_{p'}}\end{equation}
Hence we apply equality (\ref{FP3}) to $\mathtt u\star \mathtt u'$. Observing that the pairs of sequences $i_1<\dots<i_k$ and $j_1,\dots,j_k$ in (\ref{FP3}) are in a one to one correspondence with the B-diagrams $\begin{array}{c}G'\\\mcomp{i_1,\dots,i_k}{j_1,\dots,j_k}\\G\end{array}$, we obtain
\begin{equation}\begin{array}{r}
\mathtt u\star \mathtt u'=\displaystyle\sum_{G''}\left(
\prod^{(*)'} \raa_{\bullet_{p}\mathtt{seq}_{G''}(p)} \prod^{(**)'} \raa_{\langle\mathtt{seq}_{G''}\left(p'[\omega(G)]\right)}
\prod^{(***)'}\raa_{\bullet_{p}\mathtt{seq}_{G''}\left(pp'[\omega(G)]\right)}\right.\\\times\displaystyle\left.
\prod^{(*)'} \baa_{\mathtt{seq}_{G''}(p)\langle} \prod^{(**)'} \baa_{\mathtt{seq}_{G''}\left(p'[\omega(G)]\right)\diamond_{p'}}
\prod^{(***)'}\baa_{\mathtt{seq}_{G''}\left(pp'[\omega(G)]\right)\diamond_{p'}}\right),
\end{array}
\end{equation}
where the sum is over the B-diagram $G''=\begin{array}{c}G'\\\mcomp{i_1,\dots,i_k}{j_1,\dots,j_k}\\G\end{array}$ with $i_1<\dots<i_k\in H^\uparrow_f(G)$ and $j_1,\dots,j_k\in H^\downarrow_f(G)$ distinct,
the products $\displaystyle\prod^{(*)'}$ are over the paths $p\in \mathrm{Paths}(G)$ such that $p^\uparrow\not\in\{i_1,\dots,i_k\}$ and $\diamond_p=\langle$,
the products $\displaystyle\prod^{(**)'}$ are over the paths $p'\in \mathrm{Paths}(G')$ such that $p'^\downarrow\not\in\{j_1,\dots,j_k\}$ and $\bullet_{p'}=\rangle$, and
the products $\displaystyle\prod^{(***)'}$ are over the pairs of paths $p\in \mathrm{Paths}(G)$ and $p'\in \mathrm{Paths}(G')$ with $p^\uparrow=i_h$ and $p'^\downarrow=j_h$ for some $1\leq h\leq k$.\\
Notice that for a given $G''$ the paths which do not appear in the product
\begin{equation}\label{PG''}\begin{array}{rcl}
\mathcal P(G'')&=&\displaystyle\prod^{(*)'} \raa_{\bullet_p\mathtt{seq}_{G''}(p)} \prod^{(**)'} \raa_{\langle\mathtt{seq}_{G''}\left(p'[\omega(G)]\right)}
\prod^{(***)'}\raa_{\bullet_p\mathtt{seq}_{G''}\left(pp'[\omega(G)]\right)}
\\&&\displaystyle\times\prod^{(*)'} \baa_{\mathtt{seq}_{G''}(p)\langle} \prod^{(**)'} \baa_{\mathtt{seq}_{G''}\left(p'[\omega(G)]\right)\diamond_{p'}}
\prod^{(***)'}\baa_{\mathtt{seq}_{G''}\left(pp'[\omega(G)]\right)\diamond_{p'}}
\end{array}
\end{equation}
are exactly the paths $p$ of $G$ such that $\diamond_p=\rangle$ and the paths $p'[\omega(G)]$ where $p'\in\mathtt {Paths}(G')$ and $\bullet_{p'}=\langle$.
Hence, comparing (\ref{PG''}) to (\ref{wG''}), one obtains
\begin{equation}\begin{array}{r}\displaystyle
\mathtt w(G'')=\prod_{p\in\mathrm{Paths}(G)\atop p^\uparrow=\rangle}\raa_{\bullet_p \mathrm{seq}_{G}(p)}
\prod_{p'\in\mathrm{Paths}(G')\atop \bullet_{p'}=\langle}
\left(\raa_{\langle\mathrm{seq}_{G'}\left(p\right)}[|G|,\omega(G)]\right)\mathcal P(G'')
\\\displaystyle\times\prod_{p\in\mathrm{Paths}(G)\atop \diamond{p}=\rangle}\baa_{\mathrm{seq}_{G}(p)\rangle}
\prod_{p'\in\mathrm{Paths}(G')\atop \bullet_{p'}=\langle}\left(\baa_{\mathrm{seq}_{G'}\left(p'\right)\bullet_{p'}}[|G|,\omega(G)]\right)
.\end{array}
\end{equation}
So (\ref{wG*wG'}) becomes
\[
\mathtt w(G)\star \mathtt w(G')=\sum_{G''}\mathtt w(G'')
\]
where the sum is over the B-diagram $G''=\begin{array}{c}G'\\\mcomp{i_1,\dots,i_k}{j_1,\dots,j_k}\\G\end{array}$ with $i_1<\dots<i_k\in H^\uparrow_f(G)$ and $j_1,\dots,j_k\in H^\downarrow_f(G)$ distinct. In other words,
\[
\mathtt w(G)\star \mathtt w(G')=\mathtt w(G\star G').
\]
Lemma \ref{LbasisFH} allows us to conclude.
\end{proof}

\begin{example}\rm
Compare the computation in Figure \ref{star1} with Example \ref{ExFH}.
\begin{figure}[H]
\begin{center}
\begin{tikzpicture}
\setcounter{Edge}{1}
\setcounter{Vertex}{1}
\bugddux 01{b1}
\node(st) at (1,1){$\star$};
\setcounter{Edge}{1}
\setcounter{Vertex}{1}
\bugddux {1.5}1{b2}
\node(eq) at (2.5,1){$=$};
\setcounter{Edge}{1}
\setcounter{Vertex}{1}
\bugddux 31{b3}
\bugddux 41{b4}
\node(eq) at (5,1){$+$};
\setcounter{Edge}{1}
\setcounter{Vertex}{1}

\bugddux {5.5}0{b5}
\bugddux {6}{2}{b6}
\draw (d2b6) edge[in=90,out=270] (u1b5);
\node(eq) at (7,1){$+$};
\setcounter{Edge}{1}
\setcounter{Vertex}{1}

\bugddux {7.5}0{b7}
\bugddux {8}{2}{b8}
\draw (d1b8) edge[in=90,out=270] (u1b7);

\end{tikzpicture}
\end{center}
\caption{An example of computation in $\mathcal {B}$\label{star1}}

\end{figure}

Also compare Figure \ref{prod2} to
\[
\begin{array}{l}
\left(\raa_{\rangle\left(1\atop 1\right)}\raa_{\rangle\left(1\atop 2\right)}\baa_{\left(1\atop 1\right)\langle}\baa_{\left(1\atop 2\right)\langle}\right)^{\star 2}=
\raa_{\rangle\left(1\atop 1\right)}\raa_{\rangle\left(1\atop 2\right)}\raa_{\rangle\left(2\atop 3\right)}\raa_{\rangle\left(2\atop 4\right)}\baa_{\left(1\atop 1\right)\langle}\baa_{\left(1\atop 2\right)\langle}\baa_{\left(2\atop 3\right)\langle}\baa_{\left(2\atop 4\right)\langle}\\+
\raa_{\rangle\left(1\atop 1\right)\left(2\atop 3\right)}\raa_{\rangle\left(1\atop 2\right)}\raa_{\rangle\left(2\atop 4\right)}\baa_{\left(1\atop 1\right)\left(2\atop 3\right)\langle}\baa_{\left(1\atop 2\right)\langle}\baa_{\left(2\atop 4\right)\langle}+
\raa_{\rangle\left(1\atop 1\right)}\raa_{\rangle\left(1\atop 2\right)\left(2\atop 3\right)}\raa_{\rangle\left(2\atop 4\right)}\baa_{\left(1\atop 1\right)\langle}\baa_{\left(1\atop 2\right)\left(2\atop 3\right)\langle}\baa_{\left(2\atop 4\right)\langle}\\
+
\raa_{\rangle\left(1\atop 1\right)\left(2\atop 2\right)}\raa_{\rangle\left(1\atop 2\right)}\raa_{\rangle\left(2\atop 3\right)}\baa_{\left(1\atop 1\right)\left(2\atop 2\right)\langle}\baa_{\left(1\atop 2\right)\langle}\baa_{\left(2\atop 3\right)\langle}+
\raa_{\rangle\left(1\atop 1\right)}\raa_{\rangle\left(1\atop 2\right)\left(2\atop 4\right)}\raa_{\rangle\left(2\atop 3\right)}\baa_{\left(1\atop 1\right)\langle}\baa_{\left(1\atop 2\right)\left(2\atop 4\right)\langle}\baa_{\left(2\atop 3\right)\langle}\\
+
\raa_{\rangle\left(1\atop 1\right)\left(2\atop 3\right)}\raa_{\rangle\left(1\atop 2\right)\left(2\atop 4\right)}\baa_{\left(1\atop 1\right)\left(2\atop 3\right)\langle}\baa_{\left(1\atop 2\right)\left(2\atop 4\right)\langle}
+
\raa_{\rangle\left(1\atop 1\right)\left(2\atop 4\right)}\raa_{\rangle\left(1\atop 3\right)\left(2\atop 4\right)}\baa_{\left(1\atop 1\right)\left(2\atop 4\right)\langle}\baa_{\left(1\atop 2\right)\left(2\atop 3\right)\langle}
\end{array}
\]
\end{example}
\begin{figure}[h]
\begin{center}
\begin{tikzpicture}
\setcounter{Edge}{1}
\setcounter{Vertex}{1}

\bugdduu 01	{b1}
\node(star) at (1,1){$\star$};
\setcounter{Edge}{1}
\setcounter{Vertex}{1}

\bugdduu {1.5}1{b2}
\node(eq) at (2.5,1){$=$};
\setcounter{Edge}{1}
\setcounter{Vertex}{1}

\bugdduu 31	{b3}
\bugdduu 41{b4}
\node(p1) at (5,1){$+$};
\setcounter{Edge}{1}
\setcounter{Vertex}{1}
\bugdduu {5.5}0 {b5}	
\bugdduu {6}2 {b6}
\draw (d1b6) edge[in=90,out=270] (u1b5);

\node(p1) at (7,1){$+$};
\setcounter{Edge}{1}
\setcounter{Vertex}{1}
\bugdduu {7.5}0 {b7}	
\bugdduu {8}2 {b8}
\draw (d2b8) edge[in=90,out=270] (u1b7);

\node(p1) at (9,1){$+$};
\setcounter{Edge}{1}
\setcounter{Vertex}{1}
\bugdduu {9.5}0{b9}	
\bugdduu {10}2{b10}
\draw (d1b10) edge[in=90,out=270] (u2b9);

\node(p1) at (11,1){$+$};
\setcounter{Edge}{1}
\setcounter{Vertex}{1}
\bugdduu {11.5}0{b11}	
\bugdduu {12}2{b12}
\draw (d2b12) edge[in=90,out=270] (u2b11);
	
\node(p1) at (2.5,-2) {$+$};
\setcounter{Edge}{1}
\setcounter{Vertex}{1}

\bugdduu 3{-3}{b13}
\bugdduu 3{-1}{b14}
\draw (d1b14) edge[in=90,out=270] (u1b13);
\draw (d2b14) edge[in=90,out=270] (u2b13);

\node(p1) at (4,-2) {$+$};
\setcounter{Edge}{1}
\setcounter{Vertex}{1}

\bugdduu {4.5}{-3}{b15}
\bugdduu {4.5}{-1}{b16}
\draw (d2b16) edge[in=90,out=270] (u1b15);
\draw (d1b16) edge[in=90,out=270] (u2b15);

\end{tikzpicture}
\end{center}
\caption{A second example of computation in $\mathcal B$\label{prod2}}
\end{figure}

As a consequence, one has
\begin{corollary}
The algebra $\mathcal {B}$ is associative.
\end{corollary}
Alternatively, this result can be shown in a combinatorial way.
First, we extend bi-linearly the composition $\comp{a_1,\dots,a_k}{b_1,\dots,b_k}$.
Hence, we observe
\begin{equation}\label{GstarvG}
\begin{array}{c}G'\\\comp{\alpha_1,\dots,\alpha_p}{\beta_1,\dots,\beta_p}\\G \end{array}\star G''=\sum \begin{array}{c}G'\star G''\\\comp{a_1,\dots,a_k}{b_1,\dots,b_k}\\ G \end{array},%
\end{equation}
where the sum is over the sequences $1\leq a_1<\cdots<a_k\leq\omega(G) $  and the sequences of distinct elements $1\leq b_1,\dots,b_k\leq\omega(G')+\omega(G'')$ such that there exists $1\leq i_1<\cdots<i_p$ such that $a_{i_1}=\alpha_1$, $b_{i_1}=\beta_1$, $\dots$,  $a_{i_p}=\alpha_p$, $b_{i_p}=\beta_p$.\\

From (\ref{starHW2}) we obtain,
\[
(G\star G')\star G''=\sum_{1\leq \alpha_1<\dots<\alpha_p\leq \omega(G)\atop
1\leq \beta_1, \dots, \beta_p\leq \omega(G') \mbox{\tiny distinct}} \begin{array}{c}G'\\\comp{\alpha_1,\dots,\alpha_p}{\beta_1,\dots,\beta_p}\\G \end{array}\star G''
=\sum_{1\leq a_1<\dots<a_k\leq \omega(G)\atop
1\leq b_1, \dots, b_k\leq \omega(G')+\omega(G'') \mbox{\tiny distinct}
} \begin{array}{c}G'\star G''\\\comp{a_1,\dots,a_k}{b_1,\dots,b_k}\\G \end{array}=G\star (G'\star G'').
\]

\subsection{Application to the boson normal ordering}
The Heisenberg-Weyl algebra is defined as the quotient $\mathcal {HW}=\mathbb C\langle \{\aa,\aa^\dag\}\rangle/_{\mathcal J_{\mathcal {HW}}}$, where $\mathcal J_{\mathcal {HW}}$ is the ideal generated by  $\aa\aa^\dag-\aa^\dag \aa-1$. The algebra $\mathcal{HW}$  is classically related to the boson normal ordering problem.
We consider a slightly different algebra $\mathcal H$ which is obtained by adding a central element $\ee$ to $\mathcal {HW}$.
We will see that $\ee$ has a natural combinatorial interpretation.
Indeed, let us define the map $\mathfrak p:{\mathcal A}\longrightarrow \{\aa,\aa^\dag,\ee\}$ sending each element of $\mathcal A_{\rangle\times}$ to $\aa^\dag$, each element of $\mathcal A_{\times\langle}$ to $\aa$, and each element of $\mathcal A_{\langle\ \rangle}$ to $\ee$.
The map $\mathfrak p$ can be extended as linear maps  $\widetilde{\mathfrak p}:\widetilde{\mathcal F}\longrightarrow \mathcal{H}$,
$\hat{\mathfrak p}:\hat{\mathcal F}\longrightarrow \mathcal{H}$, and ${\mathfrak p}_{\mathcal H}:{\mathcal F}_{\mathcal H}\longrightarrow \mathcal{H}$.
We also define  the linear map $\mathfrak p_{\mathcal B}:\mathcal B\longrightarrow\mathcal H$ by
\begin{equation}\label{PB1}
\mathfrak p_{\mathcal B}(G)=\left(\aa^\dag\right)^{\#\{p\in \mathtt{Paths}(G):\bullet_p=\rangle\}}\aa^{\#\{p\in \mathtt{Paths}(G):\diamond_p=\langle\}}
\ee^{\#\{p\in \mathtt{Paths}(G):\bullet_p=\langle\}+\#\{p\in \mathtt{Paths}(G):\diamond_p=\rangle\}}.
\end{equation}
Equivalently, from Remark \ref{rp2h}, one has
\begin{equation}
\mathfrak p_{\mathcal B}(G)=\left(\aa^\dag\right)^{h_f^\downarrow(G)}\aa^{h_f^\uparrow(G)}
\ee^{h_c(G)}.
\end{equation}

\begin{example}\rm
If $G$ is the B-diagram of Figure \ref{bdiag1}, one has
\[
\mathfrak p_{\mathcal B}(G)=\left(\aa^\dag\right)^{3}\aa^2\ee.
\]
\end{example}

From (\ref{PB1}) it is easy to check that $\mathfrak p_{\mathcal B}$ factorizes through $\mathcal F_{\mathcal H}$. More precisely
\begin{equation}\label{p=pw}
\mathfrak p_{\mathcal B}=\mathfrak p_{\mathcal H}\circ \mathtt w.
\end{equation}
Furthermore, one has
\begin{proposition}
The maps $\widetilde{\mathfrak p}$, $\hat{\mathfrak p}$, ${\mathfrak p}_{\mathcal H}$, and $\mathfrak p_{\mathcal B}(G)$ are morphisms of algebras.
\end{proposition}
\begin{proof}
The fact that $\widetilde{\mathfrak p}$, $\hat{\mathfrak p}$, and ${\mathfrak p}_{\mathcal H}$ are morphisms is straightforward from their definitions.
The map $\mathfrak p_{\mathcal B}(G)$ is a morphism of algebras because it is the composition of two morphisms of algebras (see Formula (\ref{p=pw})).
\end{proof}

We easily check the product formula
\begin{equation}\label{multform}
\left(\aa^\dag\right)^m\aa^n\ee^q.\left(\aa^\dag\right)^r\aa^s\ee^t=\sum_{i=0}^{\min\{n,r\}}i!\binom{q}i\binom ri \left(\aa^\dag\right)^{m+r-i}\aa^{n+s-i}\ee^{q+t}.
\end{equation}
This formula has the following combinatorial interpretation. Consider two B-diagrams $G$ and $G'$, there are $i!\binom{h_f^\uparrow(G)}i\binom {h_f^\downarrow(G')}i$ ways to compose $G$ with $G'$ and obtain a B-diagram $G''$ such that $h_f^\downarrow(G'')=h_f^\downarrow(G)+h_f^\downarrow(G')-i$, $h_f^\uparrow(G'')=h_f^\uparrow(G)+h_f^\uparrow(G')-i$, and $h_c(G'')=h_c(G)+h_c(G)$.
\begin{example}\rm Compare Figure \ref{prod2} to
\[
\left(\aa^\dag\right)^2\aa^2.\left(\aa^\dag\right)^2\aa^2=\left(\aa^\dag\right)^4\aa^4+4\left(\aa^\dag\right)^3\aa^3+2\left(\aa^\dag\right)^2\aa^2.
\]
\end{example}

The relationships between the algebras defined in this section are summarized in Figure \ref{comdiag} where the dashed arrow indicates that we replace the product in $\widetilde{\mathcal F}$ by a shifted product.

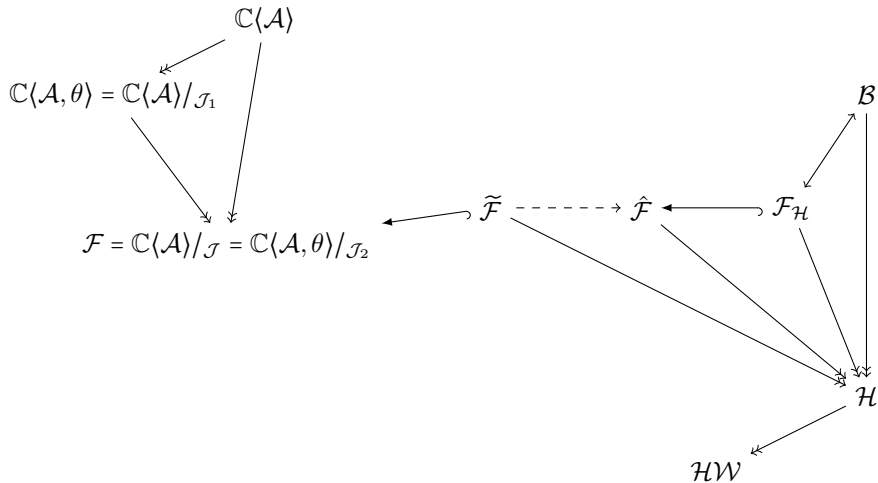
\begin{figure}[H]
\begin{center}
\begin{tikzpicture}
\node (cA) at (2,6) {$\mathbb C\langle\mathcal A\rangle$};
\node (cATheta) at (0,5) {$\mathbb C\langle\mathcal A,\theta\rangle=\mathbb C\langle\mathcal A\rangle/_{\mathcal J_1}$};
\node (F) at (1.5,3) {$\mathcal F=\mathbb C\langle\mathcal A\rangle/_{\mathcal J}=\mathbb C\langle\mathcal A,\theta\rangle/_{\mathcal J_2}$};
\path[->>] (cA) edge (cATheta);
\path[->>] (cATheta) edge (F);
\path[->>] (cA) edge (F);
\node (tF) at (5,3.5) {$\widetilde{\mathcal F}$};
\path[>=latex,right hook->] (tF) edge (F);
\node (hF) at (7,3.5) {$\hat{\mathcal F}$};
\path[dashed,<-] (hF) edge (tF);
\node (FH) at (9,3.5) {${\mathcal F}_{\mathcal H}$};
\path[>=latex,right hook->] (FH) edge (hF);
\node (B) at (10,5) {$\mathcal {B}$};
\path[<->] (FH) edge (B);
\node (H) at (10,1) {$\mathcal H$};
\node (HW) at (8,0) {$\mathcal {HW}$};
\path[->>] (B) edge (H);
\path[->>] (tF) edge (H);
\path[->>] (hF) edge (H);
\path[->>] (FH) edge (H);
\path[->>] (H) edge (HW);
\end{tikzpicture}
\end{center}
\caption{The different algebras related to the B-diagrams \label{comdiag}}

\end{figure}

\section{Cogebraic aspects of B-diagrams\label{Hopf}}
 \subsection{The Hopf algebras of B-diagrams}
 We define the linear map $\Delta:\mathcal B\longrightarrow\mathcal B\otimes\mathcal B$ by setting
\begin{equation}\label{DeltaDef1}
\Delta(G)=\sum_{I\in\mathrm{Iso}(G)}G[I]\otimes G[\complement I].\end{equation}
Equivalently, from claim \ref{isolation}, one has
\begin{equation}\label{DeltaDef2}
\Delta(G)=\sum_{\mathcal I\subset\mathrm{Connected}(G)}G\left[\bigcup_{I\in\mathcal I}I\right] \otimes 	G\left[\complement_G\bigcup_{I\in\mathcal I}I\right].\end{equation}
For simplicity we set 
\[
G\langle I\rangle=\left\{\begin{array}{ll}G[I]&\mbox{ if }I\in\mathrm{Iso}(G)\\0&\mbox{ otherwise} \end{array}\right. .
\]
With this notation, one has
\begin{equation}\label{DeltaDef3}
\Delta(G)=\sum_{I\uplus J=\llbracket 1,|G|\rrbracket}G\langle I\rangle\otimes G\langle J\rangle.\end{equation}

We also define  $\epsilon:\mathcal B\longrightarrow\mathbb C$ as the projection to $\mathcal B_0$. Obviously, $\Delta$ is a coassociative, cocomutative product and $\epsilon$ is its counity. So,
\begin{proposition}
$(\mathcal B,\Delta,\epsilon)$ is a connected graded co-commutative cogebra.
\end{proposition}
\begin{proof}
Since $\dim\mathcal B_0=1$, 
we have only to check that $\Delta$ is graded. That is $\Delta(\mathcal B_k)\subset\bigoplus_{i+j=k}\mathcal B_i\otimes\mathcal B_j$. This is straightforward from the definition of $\Delta$.
\end{proof}

Since $\mathcal B$ is a connected cogebra and an algebra with finite graded dimension, if it is a bigebra then it is a Hopf algebra. Hence, one has only to prove that $\Delta$ is a morphism of algebra.

Let us prove that for each pair $(G,G')$ of B-diagrams one has
$ \Delta(G \star G') = \Delta(G) \star \Delta(G') $. We start from the equality 
\[
\Delta(G \star G') = \sum_{I\uplus J =\llbracket1, |G|+|G'|\rrbracket} G \star G' \langle I \rangle \otimes G \star G' \langle J \rangle.
\]
Each $I$ appearing in the sum splits into two sets $(I\cap \llbracket 1, |G|\rrbracket)\uplus  (I\cap \llbracket |G| + 1,|G|+|G'|\rrbracket)$. Also, it  is  the case for the set $J$. Hence,
\[
\Delta(G \star G') = \sum_{I\uplus J = \llbracket 1, |G|\rrbracket\atop I'\uplus J'=\llbracket |G| + 1,|G|+|G'|\rrbracket} (G \star G') \langle I \cup I' \rangle \otimes (G \star G') \langle J \cup J' \rangle .
\]
Applying (\ref{starHW2}), we obtain
\[
\Delta(G \star G') 
= \sum_{1 \leqslant a_1< \cdots< a_k \leqslant \omega(G)\atop 1 \leqslant b_1, \dots, b_k \leqslant \omega(G) \mbox{\tiny distinct}} \sum_{I\uplus J=\llbracket 1, |G|\rrbracket \atop I'\uplus J'=\llbracket |G| + 1,|G|+|G'|\rrbracket} 
 \left(\begin{array}{c} G' \\\comp{a_1,\dots,a_k}{b_1, \dots,b_k} \\ G \end{array}\right)\langle I \cup I' \rangle  \otimes 
 \left(\begin{array}{c} G' \\\comp{a_1,\dots,a_k}{b_1, \dots,b_k} \\ G  \end{array}\right)\langle J \cup J' \rangle .
\]
Observe that the elements of $I$ and $J$ define two complementary subdiagrams of $G$ and the elements of $I'$ and $J'$ define two complementary subdiagrams of $G'$. Hence,
\[
\Delta(G \star G') 
= \sum_{I\uplus J=\llbracket 1, |G|\rrbracket \atop I'\uplus J'=\llbracket  1,|G'|\rrbracket}
\sum_{1 \leqslant a_1< \cdots< a_k \leqslant \omega(G)\atop 1 \leqslant b_1, \dots, b_k \leqslant \omega(G) \mbox{\tiny distinct}} 
 \left(\begin{array}{c} G' \langle I' \rangle \\\comp{[a_1,\dots,a_k]^{I}}{[b_1, \dots,b_k]^{I'}} \\ G\langle I \rangle \end{array}\right) \otimes 
 \left(\begin{array}{c} G' \langle J' \rangle \\\comp{[a_1,\dots,a_{k}]^{J}}{[b_1^{J'}, \dots,b_{k}]^{J'}} \\ G\langle J \rangle \end{array}\right) \\
\]
where the sequence $[a_1,\dots,a_k]^I$ (resp. $[a_1,\dots,a_k]^J$, $[b_1,\dots,b_k]^{I'}$ and $[b_1,\dots,b_k]^{J'}$) denotes the image of subsequence of $[a_1,\dots,a_k]$  (resp. $[a_1,\dots,a_k]$, $[b_1,\dots,b_k]$ and $[b_1,\dots,b_k]$) in $G[I]$ (resp. $G[J]$, $G[I']$ and $G[J']$). In other words, each sequence $1 \leqslant a_1< \cdots< a_k \leqslant \omega(G)$ (resp. $b_1,\dots,b_k$) splits into two subsequences one corresponding to half edges in $G[I]$ (resp. $G[I']$) and the second to half edges in $G[J]$ (resp. $G[J']$). Hence, we deduce
\[
\Delta(G \star G') 
= \sum_{I\uplus J= \llbracket 1, |G|\rrbracket \atop I'\uplus J'=\llbracket  1,|G'|\rrbracket}
\sum_{1 \leqslant a_1< \cdots< a_l \leqslant \omega(G\langle I\rangle) \atop 1 \leqslant b_1, \dots, b_l \leqslant \omega(G'\langle I'\rangle)\mbox{ \tiny distinct}} 
\sum_{1 \leqslant c_1< \dots,<c_t \leqslant \omega(G\langle J\rangle)\atop 1 \leqslant d_1, \dots, d_t \leqslant \omega(G'\langle J'\rangle)\mbox{ \tiny distinct}}
 \left(\begin{array}{c} G' \langle I' \rangle \\\comp{a_1,\dots,a_l}{b_1, \dots,b_l} \\ G\langle I \rangle \end{array}\right) \otimes 
 \left(\begin{array}{c} G' \langle J' \rangle \\\comp{c_1,\dots,c_t}{d_1, \dots,d_t} \\ G\langle J \rangle \end{array}\right). 
\]
Finally, one computes
\begin{align*}
\Delta(G \star G') 
 &= \sum_{I\uplus J= \llbracket 1, |G|\rrbracket\atop I'\uplus J'=\llbracket  1,|G'|\rrbracket}
\left(G\langle I \rangle \star G'\langle I' \rangle\right) \otimes \left(G\langle J \rangle \star G'\langle J' \rangle\right)\\
 &=\left(\sum_{I\uplus J= \llbracket 1, |G|\rrbracket } G\langle I \rangle \otimes G\langle J \rangle \right) \star
  \left(\sum_{I'\uplus J'= \llbracket 1,  |G'|\rrbracket} G'\langle I' \rangle \otimes G'\langle J' \rangle \right) \\
 &= \Delta(G) \star \Delta(G').
\end{align*}
This shows that $\mathcal B$ is a graded bialgebra with finite dimensional graded component. Hence,
\begin{theorem}
$(\mathcal B,\star,\Delta)$ is a graded Hopf algebra.
\end{theorem}
\subsection{Primitive elements}
It is classical and easy to show that the primitive elements endowed with the bracket product $[P,Q]=P\star Q-Q\star P$ form a Lie algebra. The Cartier-Quillen-Milnor-Moore  Theorem (see e.g., \cite{MM}) assures that each graded, connected, and cocommutative Hopf algebra is isomorphic to the universal enveloping of the Lie algebra of its primitive elements. Obviously, it is the case for $\mathcal B$. More precisely, $\mathcal B$ is a Hopf algebra which is graded by the number of half edges of the B-diagrams and the dimensions of its graded components are finite. Since $\mathcal B$ is free on the indivisible B-diagrams, the dimensions of the graded components of the space $\mathtt{Prim}(\mathcal B)$ of the primitive elements are necessarily the same than those of the free Lie algebra  $\mathcal L(\mathbb G)$ generated by the indivisible elements.

Now, let us  recall a few facts about the Eulerian idempotent. The convolution product is classically defined on $Hom(\mathcal B,\mathcal B)$ by $f* g=\mu\circ(f\otimes g)\circ \Delta$ where $\mu$ denotes the linear map from $\mathcal B\otimes \mathcal B$ to $\mathcal B$ sending $P\otimes Q$ to $P\star Q$. 
The Eulerian idempotent is defined by (see e.g., \cite{Reutenauer})
\begin{equation}\label{euler}
\pi_1:=\log_*(Id)=\sum_{k>0}{(-1)^{k+1}\over k}(Id-\xi)^{*k}
\end{equation}
where $\xi$ denotes the unity of the convolution algebra, that is the projection on the space generated by the empty B-diagram $\varepsilon$.
Remarking that $(Id-\xi)^{*k}$ sends to $0$ any B-diagram that can be written as $G_1|\cdots|G_k$ with each $G_i\neq\varepsilon$, $\pi_1$ maps each B-diagram to a finite linear combination of B-diagrams which is known to belong in $\mathtt{Prim}(\mathcal B)$. 
\begin{example}\rm
Examine the following example 
\begin{center}
\begin{tikzpicture}
\setcounter{Edge}{1}
\setcounter{Vertex}{1}
\node (bla) at (-1,1) {\Large $B=$};
\bugddux 00{b1} 
\bugddux 10{b2}
\bugddux 02{b3}
\node (bla) at (2.3,1) {\Large$\displaystyle\mathop{\longrightarrow}^{\pi_1}$};
\draw (d1b3) edge[in=90,out=270]  (u1b1);
\setcounter{Edge}{1}
\setcounter{Vertex}{1}
\bugddux 30{b1} 
\bugddux 40{b2}
\bugddux 32{b3}
\draw (d1b3) edge[in=90,out=270]  (u1b1);
\node (bla) at (5,1) {\Large$-\frac12$};
\setcounter{Edge}{1}
\setcounter{Vertex}{1}
\bugddux {5.5}0{b1} 
\bugddux {5.5}2{b3}
\bugddux {6.5}0{b2}
\draw (d1b3) edge[in=90,out=270]  (u1b1);
\node (bla) at (7.5,1) {\Large$-\frac12$};
\setcounter{Edge}{1}
\setcounter{Vertex}{1}
\bugddux {8}0{b2}
\bugddux {9}0{b1} 
\bugddux {9}2{b3}

\draw (d1b3) edge[in=90,out=270]  (u1b1);
\node (bla) at (10.5,1) {\Large $+C$};

\end{tikzpicture}
\end{center}
where $C$ is a linear combination of connected B-diagrams.
Setting $\hat\Delta=\Delta-Id\otimes \varepsilon-\varepsilon\otimes Id$, we observe the three B-diagrams occurring in the above formula  have the same image by $\hat\Delta$
\begin{center}
\begin{tikzpicture}
\setcounter{Edge}{1}
\setcounter{Vertex}{1}
\bugddux {0}0{b1} 
\bugddux {0}2{b2}
\draw (d1b2) edge[in=90,out=270]  (u1b1);
\node (bla) at (1.25,1) {\Large$\otimes$};
\setcounter{Edge}{1}
\setcounter{Vertex}{1}
\bugddux {2}0{b1} 
\node (bla) at (3,1) {\Large$+$};
\setcounter{Edge}{1}
\setcounter{Vertex}{1}
\bugddux {4}0{b1} 
\node (bla) at (5.25,1) {\Large$\otimes$};
\setcounter{Edge}{1}
\setcounter{Vertex}{1}
\bugddux {6}0{b1} 
\bugddux {6}2{b2}
\draw (d1b2) edge[in=90,out=270]  (u1b1);
\end{tikzpicture}
\end{center}
Hence, $\pi_1(B)$ is primitive.
\end{example}
A fast examination of $\pi_1(G)$, where $G$ is a indivisible $B$-diagram, shows that
$
\pi_1(G)=G+\cdots
$, where $\cdots$ means a linear combination of $B$-diagrams which are non indivisible or have less connected components than $G$. Since, the indivisible $B$-diagrams are algebraically independent, this shows that $\mathtt{Prim}(\mathcal B)$ contains a subalgebra which is isomorphic to $\mathcal L(\mathbb G)$. Hence, the dimensions of the graded components of $\mathtt{Prim}(\mathcal B)$ being the same than those of $\mathcal L(\mathbb G)$, we deduce the following result.
\begin{theorem}
The Lie algebra of the primitive elements of $\mathcal B$ is isomorphic to the free Lie algebra generated by the indivisible B-diagrams. Indeed, $\mathrm{Prim}(\mathcal B)$ is freely generated as a Lie algebra by $\{\pi_1(G):G\in\mathbb G\}$.
\end{theorem} 

\section{Two subalgebras\label{subalgebras}}
\subsection{Word symmetric functions\label{WSym}}
In this section, we investigate a combinatorial Hopf algebra whose bases are indexed by set partitions. A \emph{set partitions} is a partition of $\llbracket 1,n\rrbracket$ for a given $n\in \N$. Let $\pi\vDash n$ denote the fact that $\pi$ is a set partition of $\llbracket1,n\rrbracket$ and define $\pi\uplus\pi':=\pi\cup \{\{i_1+n,\dots,i_k+n\}:\{i_1,\dots,i_k\}\in\pi'$ for $\pi \vDash n$. The set of set partitions is endowed with the partial order $\preceq$ of refinement defined by $\pi\preceq\pi'$ if the sets of $\pi'$ are the union of sets of $\pi$. Finally, we say that a set partition $\pi$ is  \emph{indivisible} if $\pi=\pi'\uplus\pi''$ where $\pi'$ and $\pi''$ are set partitions implies either $\pi'=\pi$ or $\pi''=\pi$. 
 
The algebra of word symmetric functions was introduced by Wolf \cite{Wolf} as a noncommutative analogue of the algebra of symmetric functions. The Hopf structure of this algebra was  described in \cite{BZ,BHRZ} (see also \cite{bergeron2005invariants} for the polynomial realization with finite alphabets). In order to avoid confusion with some other  analogues of symmetric functions (see e.g. \cite{ComHopf}), we denote this algebra by $\WSym$ which is described (as an abstract bigebra) as the space spanned by the elements $M_\pi$, where $\pi$ is a set partitions, and  endowed with the product
\begin{equation}\label{wsymproduct}
M_\pi M_{\pi'}=\sum_{\pi\uplus\pi'\preceq\pi''\preceq\{\{1,\dots,n\},\{n+1,\dots,n'+1\}\}}M_{\pi''},
\end{equation}  
if $\pi\vDash n$ and $\pi'\vDash n'$,
 and the coproduct
\begin{equation}\label{wsymcoproduct}
\Delta(M_\pi)=\sum_{\pi=e\uplus f}M_{\std(e)}\otimes M_{\std(f)},
\end{equation}
with $\std(e)=\{\{\phi(i_1),\dots,\phi(i_k)\}:\{i_1,\dots,i_k\}\in e\}$  where $\phi$ is the unique increasing bijection from $\bigcup_{\alpha\in e}\alpha$ to $\llbracket 1,\sum_{\alpha\in e}\#\alpha\rrbracket$. The following result is well known and is an easy consequence of (\ref{wsymproduct}).
\begin{proposition}\label{WSymFree}
The algebra $\WSym$ is freely generated as an algebra by $$\{M_\pi:\pi\mbox{ is an indivisible set partition}\}.$$
\end{proposition} 
To each set partition $\pi\vDash n$, we associate a B-diagram $b_\pi=(n,[1,\dots,1],\llbracket 1,n\rrbracket,\llbracket 1,n\rrbracket,E_\pi)$ where $E_\pi$ is the set of the pairs $(i,j)$ such that $i,j\in e\in \pi$, and $j=\{\min(\ell)\in e: i<\ell\}$. Graphically, the components of $\pi$ corresponds to the connected sub B-diagrams of $b_\pi$.
\begin{example}\rm Consider in Figure \ref{fbpi}, the graphical representation of 
\[b_{\{\{1,3\},\{2\},\{4,7,8\},\{5,6\}\}}=(8,[1,1,1,1,1,1,1],\llbracket 1,8\rrbracket,\llbracket 1,8\rrbracket,\{(1,3),(4,7),(7,8),(5,6)\}).\]
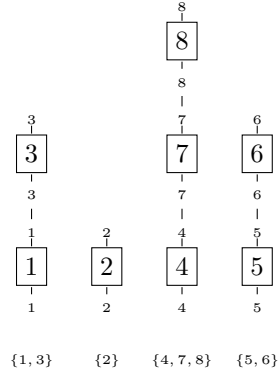
\begin{figure}[h]
\begin{center}
\begin{tikzpicture}
\setcounter{Edge}{1}
\setcounter{Vertex}{1} 
\bugdu 01{b1}
\bugdu 11{b2}
\bugdu 0{2.5}{b3}
\bugdu 21{b4}
\bugdu 31{b5}
\bugdu 3{2.5}{b6}
\bugdu 2{2.5}{b7}
\bugdu 24{b8}
\draw (d1b3) edge[in=90,out=270] (u1b1);
\draw (d1b8) edge[in=90,out=270] (u1b7);
\draw (d1b7) edge[in=90,out=270] (u1b4);
\draw (d1b6) edge[in=90,out=270] (u1b5);
\node (t1) at (0.2,0) {\tiny$\{1,3\}$};
\node (t2) at (1.2,0) {\tiny$\{2\}$};
\node (t3) at (2.2,0) {\tiny$\{4,7,8\}$};
\node (t4) at (3.2,0) {\tiny$\{5,6\}$};
\end{tikzpicture}
\end{center}
\caption{Graphical representation of $b_{\{\{1,3\},\{2\},\{4,7,8\},\{5,6\}\}}$\label{fbpi}}
\end{figure} 
\end{example}
Obviously, the space $\mathcal B^1_1$ generated by the elements $b_\pi$ is stable by product and $\Delta$ sends $\mathcal B^1_1$ to $\mathcal B^1_1\otimes \mathcal B^1_1$. So, $\mathcal B^1_1$ is a sub Hopf algebra of $\mathcal B$. More precisely, one checks
\begin{equation}\label{bprod}
b_\pi \star b_{\pi'}=\sum_{\pi\uplus\pi'\preceq\pi''\preceq\{\{1,\dots,n\},\{n+1,\dots,n'+1\}\}}b_{\pi''},
\end{equation}
if $\pi\vDash n$ and $\pi'\vDash	 n'$,
and
\begin{equation}\label{bcoproduct}
\Delta(b_\pi)=\sum_{\pi=e\uplus f}b_{\std(e)}\otimes b_{\std(f)}.
\end{equation}
Remarking that $\pi$ is indivisible if and only if the B-diagram $b_\pi$ is indivisible, we deduce that $\mathcal B^1_1$ is the free subalgebra of $\mathcal B$ generated by the set $\{b_\pi:\pi\ \mbox{ is an indivisible set partition}\}$.
Hence, comparing (\ref{bprod}) to (\ref{wsymproduct}) and (\ref{bcoproduct}) to (\ref{wsymcoproduct}), Proposition \ref{WSymFree} implies the following result.
\begin{theorem}
The subalgebra $\mathcal B^1_1$ generated by $\{b_\pi:\pi\ \mbox{ is an indivisible set partition}\}$ is a Hopf algebra isomorphic to $\WSym$.
\end{theorem}

\begin{example}\rm
For instance, compare
\[ \begin{array}{rcl}M_{\{\{1,3\},\{2\}\}} M_{\{\{1\},\{2\}\}} &=& M_{\{\{1,3\},\{2\},\{4\},\{5\}\}} + M_{\{\{1,3\},\{2,4\},\{5\}\}}
+ 
M_{\{\{1,3,4\},\{2\},\{5\}\}}+
M_{\{\{1,3\},\{2,5\},\{4\}\}}\\&&
+M_{\{\{1,3,5\},\{2\},\{4\}\}}
+
M_{\{\{1,3,4\},\{2,5\}\}}
+
M_{\{\{1,3,5\},\{2,4\}\}},
\end{array}
\]
to the product pictured in Figure \ref{wprod}.
\begin{figure}[h]
\begin{center}
\begin{tikzpicture}
\setcounter{Edge}{1}
\setcounter{Vertex}{1} 
\bugdu 00{b1}
\bugdu {0.5}0{b2}
\bugdu 0{1.5}{b3}
\draw (d1b3) edge[in=90,out=270] (u1b1);
\node (star) at (1.2,0.5) {$\star$};
\setcounter{Edge}{1}
\setcounter{Vertex}{1} 
\bugdu {1.5}0{b1}
\bugdu {2}0{b2}
\node (star) at (2.7,0.5) {$=$};
\setcounter{Edge}{1}
\setcounter{Vertex}{1} 
\bugdu 30{b1}
\bugdu {3.5}0{b2}
\bugdu {3}{1.5}{b3}
\bugdu {4}0{b4}
\bugdu {4.5}0{b5}
\draw (d1b3) edge[in=90,out=270] (u1b1);
\node (p1) at (5.2,0.5) {$+$};
\setcounter{Edge}{1}
\setcounter{Vertex}{1} 
\bugdu {5.5}0{b1}
\bugdu {6}0{b2}
\bugdu {5.5}{1.5}{b3}
\bugdu {5.5}3{b4}
\bugdu {6.5}0{b5}
\draw (d1b3) edge[in=90,out=270] (u1b1);
\draw (d1b4) edge[in=90,out=270] (u1b3);
\node (p1) at (7.2,0.5) {$+$};
\setcounter{Edge}{1}
\setcounter{Vertex}{1} 
\bugdu {7.5}0{b1}
\bugdu {8}0{b2}
\bugdu {7.5}{1.5}{b3}
\bugdu {8}{1.5}{b4}
\bugdu {8.5}0{b5}
\draw (d1b3) edge[in=90,out=270] (u1b1);
\draw (d1b4) edge[in=90,out=270] (u1b2);
\node (p1) at (9.2,0.5) {$+$};
\setcounter{Edge}{1}
\setcounter{Vertex}{1} 
\bugdu {9.5}0{b1}
\bugdu {10}0{b2}
\bugdu {9.5}{1.5}{b3}
\bugdu {10.5}{0}{b4}
\bugdu {9.5}{3}{b5}
\draw (d1b3) edge[in=90,out=270] (u1b1);
\draw (d1b5) edge[in=90,out=270] (u1b3);
\node (p1) at (11.2,0.5) {$+$};
\setcounter{Edge}{1}
\setcounter{Vertex}{1} 
\bugdu {11.5}0{b1}
\bugdu {12}0{b2}
\bugdu {11.5}{1.5}{b3}
\bugdu {12.5}{0}{b4}
\bugdu {12}{1.5}{b5}
\draw (d1b3) edge[in=90,out=270] (u1b1);
\draw (d1b5) edge[in=90,out=270] (u1b2);
\node (p1) at (13.2,0.5) {$+$};
\setcounter{Edge}{1}
\setcounter{Vertex}{1} 
\bugdu {13.5}0{b1}
\bugdu {14}0{b2}
\bugdu {13.5}{1.5}{b3}
\bugdu {13.5}{3}{b4}
\bugdu {14}{1.5}{b5}
\draw (d1b3) edge[in=90,out=270] (u1b1);
\draw (d1b5) edge[in=90,out=270] (u1b2);
\draw (d1b4) edge[in=90,out=270] (u1b3);
\node (p1) at (14.7,0.5) {$+$};
\setcounter{Edge}{1}
\setcounter{Vertex}{1} 
\bugdu {15}0{b1}
\bugdu {15.5}0{b2}
\bugdu {15}{1.5}{b3}
\bugdu {15.5}{1.5}{b4}
\bugdu {15}{3}{b5}
\draw (d1b3) edge[in=90,out=270] (u1b1);
\draw (d1b5) edge[in=90,out=270] (u1b3);
\draw (d1b4) edge[in=90,out=270] (u1b2);

\end{tikzpicture}
\end{center}
\caption{An example of product in $\mathcal B^1_1$ \label{wprod}}
\end{figure}
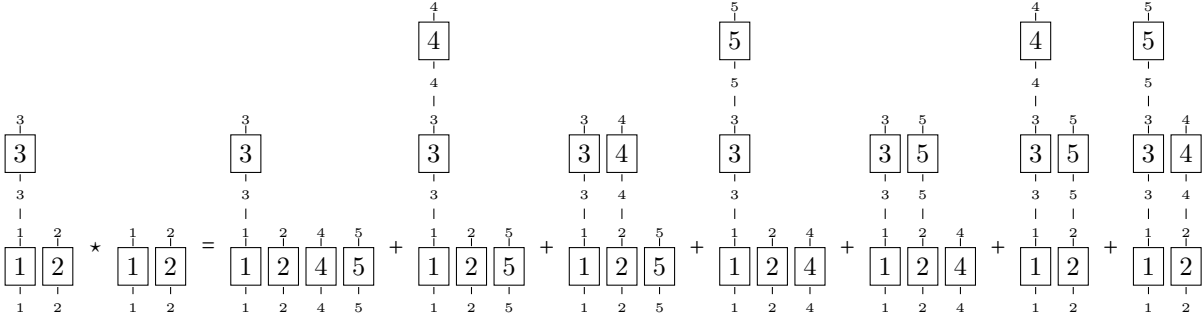 
\end{example}

\subsection{Bi-words symmetric functions\label{BWSym}}
In this section, we consider the Hopf algebra $\BWSym$ of set partitions into lists defined in \cite{BCLMA,ABCLM}. A set partition into lists is a set  of lists $\Pi=\{[i^1_1,\dots,i^1_{\ell_1}],\dots,[i^k_1,\dots,i^k_{\ell_k}]\}$ satisfying $k\geq 0$, $\ell_j\leq 1$ for each $1\leq j\leq k$, the integers $i^1_1,\dots,i^1_{\ell_1},\dots,i^k_1,\dots,i^k_{\ell_k}$ are distinct, and $\llbracket 1,n\rrbracket=\{i^j_t:1\leq j\leq k, 1\leq t\leq \ell_j\}$ for a certain $n\in\N$; we let $\Pi\vDash n$ denote this property. As for the set partitions we define $\Pi\uplus\Pi'=\Pi\cup\{[i_1+n,\dots,i_k+n]:[i_1,\dots,i_k]\in\Pi'\}$ for $\Pi\vDash n$ and we will say that $\Pi$ is \emph{indivisible} if $\Pi=\Pi'\uplus\Pi''$ for some set partitions in lists $\Pi', \Pi''$ implies either $\Pi'=\Pi$ or $\Pi''=\Pi$.

The underlying space of $\BWSym$ is freely generated by the set $\{\Phi^\Pi:\Pi\vDash n, n\geq 0\}$. Whence endowed with the product
\begin{equation}\label{bwsymproduct}
\Phi^\Pi\Phi^{\Pi'}=\Phi^{\Pi\uplus\Pi'}
\end{equation}
and the coproduct
\begin{equation}\label{bwsymcoproduct}
\Delta(\Phi^\Pi)=\sum_{\Pi=E\uplus F}\Phi^{\std(E)}\otimes \Phi^{\std(F)},
\end{equation}
with $\std(E)=\{[\phi(i_1),\dots,\phi(i_k)]: [i_1,\dots,i_k]\in e\}$  where $\phi$ is the unique increasing bijection from $\bigcup_{[\alpha_1,\dots,\alpha_p]\in E}\{\alpha_1,\dots,\alpha_p\}$ to $\llbracket 1,\sum_{[\alpha_1,\dots,\alpha_p]\in E}p\rrbracket$, $\BWSym$ is a Hopf algebra.

We need the following property which is straightforward from (\ref{bwsymproduct}).
\begin{proposition}\label{bwsymfree}
$\BWSym$ is freely generated as an algebra by the set
$
\{\Phi^\Pi:\Pi\ \mbox{ is indivisible}\}.
$
\end{proposition}

Let $\mathcal B^2_1$ be the subspace of $\mathcal B$ generated by the set $\mathcal G^2_1$ of B-diagrams under the form $$(n,[2,\dots,2],\llbracket 1,2n\rrbracket,\{1,3,5,\dots,2n-1\},E).$$ Remarking that $\mathcal G^2_1$ is exactly the set of the B-diagrams whose each vertex has two outer non cut half edges, one inner non cut half edge and one inner cut half edge, the space $\mathcal B^2_1$ is stable for the product $\star$ and $\Delta$ sends $\mathcal B^2_1$   to $\mathcal B^2_1\otimes \mathcal B^2_1$. In other words, $\mathcal B^2_1$ is a sub Hopf algebra of $\mathcal B$.

Let us describe a graded one-to-one correspondence between the set partitions into lists and the B-diagrams of $\mathcal G^2_1$. 
First we associate to each permutation $\sigma\in \S_n$ a connected B-diagram $m_\sigma=(n,[2,\dots,2],\llbracket 1,2n\rrbracket,\{1,3,5,\dots,2n-1\},E_\sigma)$ where $(2i-1,2j-1)\in E_\sigma$ if and only if $i<j$ and $j=\min\{k:\alpha_i<\sigma^{-1}(k)<\sigma^{-1}(i)\}$ where $\alpha_i=\sup\{\sigma^{-1}(k):k<i,\sigma^{-1}(k)<\sigma^{-1}(i)\}$ and 
$(2i,2j-1)\in E_\sigma$ if and only if $i<j$ and $j=\min\{k:\sigma^{-1}(i)<\sigma^{-1}(k)<\beta_i\}$  where $\beta_i=\inf\{\sigma^{-1}(k):k<i,\sigma^{-1}(k)>\sigma^{-1}(i)\}$.
\begin{example}\rm 
Consider the permutation $\sigma=[5,2,4,1,3,7,6]$. We have $\alpha_1=\alpha_2=\alpha_5=-\infty$, $\alpha_3=4$, $\alpha_4=2$, $\alpha_6=\alpha_7=5$, $\beta_1=\beta_3=\beta_6=+\infty$, $\beta_2=\beta_4=4$, $\beta_5=2$, and $\beta_7=7$. Hence,
\[
m_\sigma=(7,[2,2,2,2,2,2,2],\llbracket 1,14\rrbracket,\{1,3,5,7,9,11,13\},\{(1,3),(2,5),(3,9),(4,7),(6,11),(11,13)\}).
\]
See Figure \ref{msigma} for a graphical representation.
\begin{figure}[h]
\begin{center}
\begin{tikzpicture}
\setcounter{Edge}{1}
\setcounter{Vertex}{1} 
\bugdxuu 40{b1}
\bugdxuu {2}{1.5}{b2}
\bugdxuu {6}{1.5}{b3}
\bugdxuu {3}{3}{b4}
\bugdxuu {1}{3}{b5}
\bugdxuu {8}{3}{b6}
\bugdxuu {7}{4.5}{b7}
\draw (d1b2) edge[in=90,out=270] (u1b1);
\draw (d1b3) edge[in=90,out=270] (u2b1);
\draw (d1b5) edge[in=90,out=270] (u1b2);
\draw (d1b4) edge[in=90,out=270] (u2b2);
\draw (d1b6) edge[in=90,out=270] (u2b3);
\draw (d1b7) edge[in=90,out=270] (u1b6);
\node (n1) at (1.2,-0.7) {$5$};
\node (n1) at (2.2,-0.7) {$2$};
\node (n1) at (6.2,-0.7) {$3$};
\node (n1) at (4.2,-0.7) {$1$};
\node (n1) at (3.2,-0.7) {$4$};
\node (n1) at (8.2,-0.7) {$6$};
\node (n1) at (7.2,-0.7) {$7$};
\draw (1.7,-1) -- (1.7,5.5);
\draw (2.7,-1) -- (2.7,5.5);
\draw (3.7,-1) -- (3.7,5.5);
\draw (5.7,-1) -- (5.7,5.5);
\draw (6.7,-1) -- (6.7,5.5);
\draw (7.7,-1) -- (7.7,5.5);
\end{tikzpicture}
\end{center}
\caption{The B-diagram $m_{[5,2,4,1,3,7,6]}$ \label{msigma}}
\end{figure}
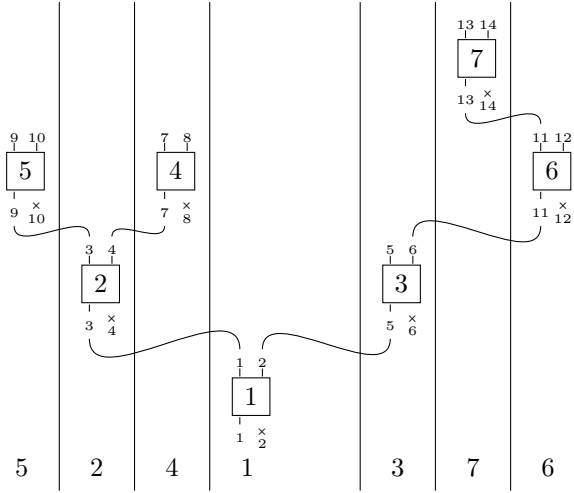 
\end{example}
 Noticing that any connected B-diagram $G$ in $\mathcal G^2_1$  satisfies that $h^\uparrow_f(G)=|G|+1$, we  construct $n+1$ different B-diagrams by adding one vertex to $G$. A quick reasoning by induction on the number of vertices shows that the set of the connected B-diagram in  $\mathcal G^2_1$ with $n$ vertices is exactly the set of the B-diagrams $m_\sigma$ for $\sigma\in\S_n$ and the correspondence $\sigma\longrightarrow m_\sigma$ is one to one.\\
 We extend the construction to the set partitions into lists as follows. To each integer sequence $I=[i_1,\dots,i_n]$ of $n$ distinct integers, we associate the permutation $\std(I)$ obtained by replacing each $j_k$ by $k$ in $I$ for each $k$ where $j_1<\dots<j_n$ and $\{j_1,\dots,j_n\}=\{i_1,\dots,i_n\}$. To each set partition into lists $\Pi=\{L_1,\dots,L_k\}$, we associate the unique B-diagram $m_\Pi$ in $\mathcal G^2_1$ with $k$ connected components obtained by replacing the integers $\ell$ in each $m_{\std(L_t)}$ by $j^t_\ell$ in $L_t$ where $L_t=[i^t_1,\dots,i^t_{n_t}]$, $j^t_1<\dots<j^t_{n_t}$ and $\{i_1^t,\dots,i_{n_t}^t\}=\{j^t_1<\dots<j^t_{n_t}\}$. Hence,  $\mathcal G^2_1=\{m_\Pi:\Pi\ \mbox{ is a set partition into lists}\}$ and the correspondence $\Pi\longrightarrow m_\Pi$ is one to one. Furthermore, by construction, $\Pi$ is indivisible if and only if $m_\Pi$ is indivisible.
 
 We deduce from the above discussion that the algebra $\mathcal B^2_1$ has the same graded dimension than $\BWSym$ and that it is freely generated as a an algebra by the set $$\{m_\Pi:\Pi\ \mbox{ is an indivisible set partition into lists}\}.$$

Also, notice  that (\ref{starHW}) allows us to interpret the product $m_\Pi\star m_{\Pi'}$ in terms of set partitions in lists. Let $\Pi=[L_1,\dots,L_k]\vDash n$ and $\Pi'=[L'_1,\dots,L'_{k'}]\vDash n'$ be two set partitions into lists. One has
\begin{equation}\label{prodmPi}m_\Pi\star m_{\Pi'}=m_{\Pi\uplus\Pi'}+\sum m_{\Pi''}\end{equation}
where the sum is over the set partitions in list $\Pi''=[L''_1,\dots,L''_{k''}]\vDash n+n'$ such that each $L''_{i}$ is 
\begin{enumerate}
\item either a list $L_j$,
\item or a list $[j_1+n,\dots,j_\ell+n]$ with $[j_1,\dots,j_\ell]\in\Pi'$,
\item or a list $[i_1,\dots,i_{p_1},j^1_1+n,\dots,j^1_{\ell_1}+n,i_{p_1}+1,\dots,i_{p_2},
j^2_1+n,\dots,j^2_{\ell_2}+n,\dots,i_{p_{t-1}}+1,\dots,i_{p_t},j^t_1+n,\dots,j^1_{\ell_t}+n,i_{p_t}+1,\dots,i_{p_{t+1}}]$ where $t>1$, $p_1<p_2<\dots<p_{t+1}$, $[i_1,\dots,i_{p_{t+1}}]\in \Pi$, and for each $1\leq s\leq t$, $[j^s_1,\dots,j^s_{\ell_s}]\in \Pi''$.
\end{enumerate}
\begin{example}\label{exempleB2}\rm
For instance, compare
\[ \begin{array}{rcl}m_{\{[3,1],[2]\}} m_{\{[1,2]\}} 
&=& m_{\{[3,1],[2],[4,5]\}} + m_{\{[3,1],[2,4,5]\}} + m_{\{[3,1],[4,5,2]\}} \\&& 
+ m_{\{[3,1,4,5],[2]\}} + m_{\{[3,4,5,1],[2]\}} + m_{\{[4,5,3,1],[2]\}}.
\end{array}
\]
to the product pictured in Figure \ref{bwprod}.
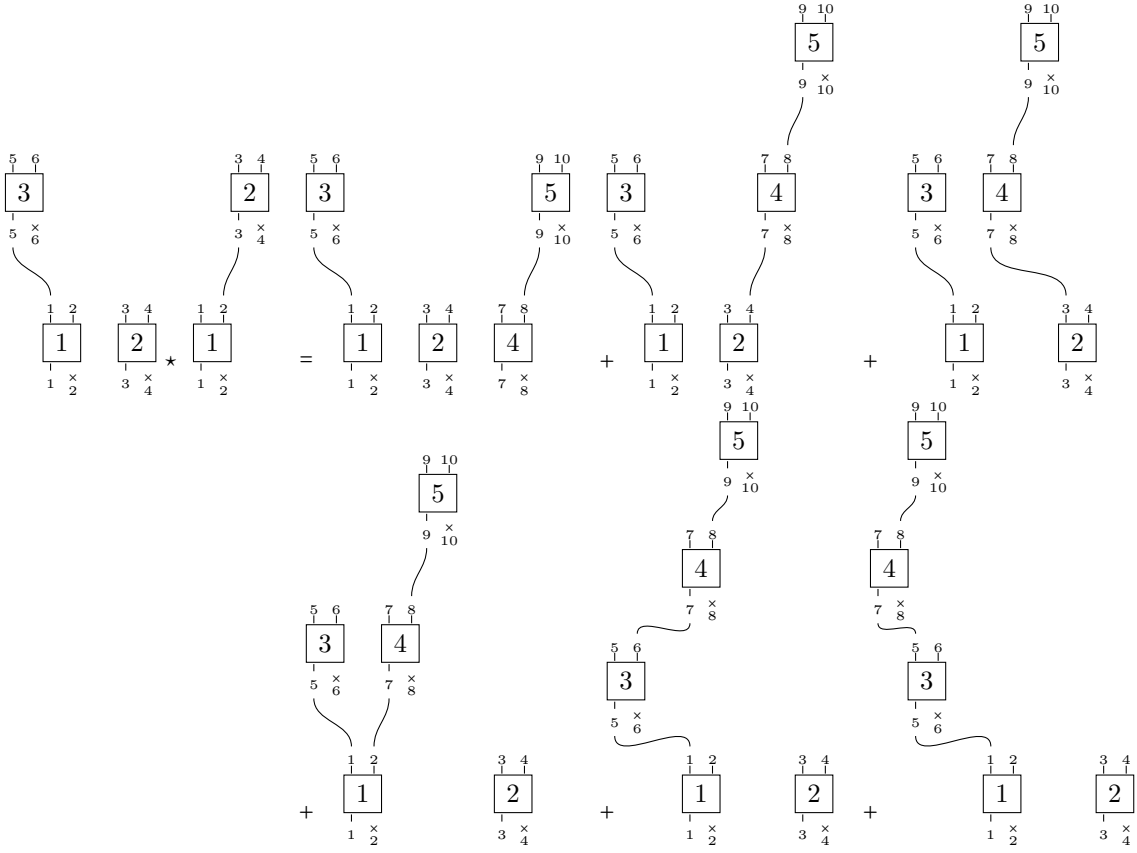
\begin{figure}[h]
\begin{center}
\begin{tikzpicture}
\setcounter{Edge}{1}
\setcounter{Vertex}{1} 
\bugdxuu {1.5}{1}{b1}
\bugdxuu {2.5}{1}{b2}
\bugdxuu {1}{3}{b3}
\draw (d1b3) edge[in=90,out=270] (u1b1);
\node (star) at (3.2,1) {$\star$};
\setcounter{Edge}{1}
\setcounter{Vertex}{1} 
\bugdxuu {3.5}{1}{b1}
\bugdxuu {4}{3}{b2}
\draw (d1b2) edge[in=90,out=270] (u2b1);
\node (star) at (5,1) {$=$};
\setcounter{Edge}{1}
\setcounter{Vertex}{1} 
\bugdxuu {5.5}{1}{b1}
\bugdxuu {6.5}{1}{b2}
\bugdxuu {5}{3}{b3}
\bugdxuu {7.5}{1}{b4}
\bugdxuu {8}{3}{b5}
\draw (d1b3) edge[in=90,out=270] (u1b1);
\draw (d1b5) edge[in=90,out=270] (u2b4);
\node (p1) at (9,1) {$+$};
\setcounter{Edge}{1}
\setcounter{Vertex}{1} 
\bugdxuu {9.5}{1}{b1}
\bugdxuu {10.5}{1}{b2}
\bugdxuu {9}{3}{b3}
\bugdxuu {11}{3}{b4}
\bugdxuu {11.5}{5}{b5}
\draw (d1b3) edge[in=90,out=270] (u1b1);
\draw (d1b4) edge[in=90,out=270] (u2b2);
\draw (d1b5) edge[in=90,out=270] (u2b4);
\node (p1) at (12.5,1) {$+$};
\setcounter{Edge}{1}
\setcounter{Vertex}{1} 
\bugdxuu {13.5}{1}{b1}
\bugdxuu {15}{1}{b2}
\bugdxuu {13}{3}{b3}
\bugdxuu {14}{3}{b4}
\bugdxuu {14.5}{5}{b5}
\draw (d1b3) edge[in=90,out=270] (u1b1);
\draw (d1b4) edge[in=90,out=270] (u1b2);
\draw (d1b5) edge[in=90,out=270] (u2b4);

\node(p1) at (5,-5) {$+$};
\setcounter{Edge}{1}
\setcounter{Vertex}{1} 
\bugdxuu {5.5}{-5}{b1}
\bugdxuu {7.5}{-5}{b2}
\bugdxuu {5}{-3}{b3}
\bugdxuu {6}{-3}{b4}
\bugdxuu {6.5}{-1}{b5}
\draw (d1b3) edge[in=90,out=270] (u1b1);
\draw (d1b4) edge[in=90,out=270] (u2b1);
\draw (d1b5) edge[in=90,out=270] (u2b4);

\node(p1) at (9,-5) {$+$};
\setcounter{Edge}{1}
\setcounter{Vertex}{1} 
\bugdxuu {10}{-5}{b1}
\bugdxuu {11.5}{-5}{b2}
\bugdxuu {9}{-3.5}{b3}
\bugdxuu {10}{-2}{b4}
\bugdxuu {10.5}{-0.3}{b5}
\draw (d1b3) edge[in=90,out=270] (u1b1);
\draw (d1b4) edge[in=90,out=270] (u2b3);
\draw (d1b5) edge[in=90,out=270] (u2b4);
\node (p1) at (12.5,-5) {$+$};
\setcounter{Edge}{1}
\setcounter{Vertex}{1} 
\bugdxuu {14}{-5}{b1}
\bugdxuu {15.5}{-5}{b2}
\bugdxuu {13}{-3.5}{b3}
\bugdxuu {12.5}{-2}{b4}
\bugdxuu {13}{-0.3}{b5}
\draw (d1b3) edge[in=90,out=270] (u1b1);
\draw (d1b4) edge[in=90,out=270] (u1b3);
\draw (d1b5) edge[in=90,out=270] (u2b4);
\end{tikzpicture}
\end{center}
\caption{An example of product in $\mathcal B^2_1$ \label{bwprod}}
\end{figure} 
\end{example}
We deduce the following result.

\begin{theorem}
The Hopf algebras $\BWSym$ and $\mathcal B^2_1$ are isomorphic.
\end{theorem}

\begin{proof}
From the above discussion and Proposition  \ref{bwsymfree}, the algebras $\BWSym$ and $\mathcal B^2_1$ are isomorphic. An explicit isomorphism $\eta$ sends  $\Phi^\Pi$ to $m_\Pi$ for each nonsplitable set partition into lists $\Pi$. It remains to prove that it is  a morphism of cogebras. Remarking that the lists of a set partition into lists $\Pi$ correspond to the connected components of $m_\Pi$, we show that equality (\ref{DeltaDef2}) implies  \begin{equation}\label{eqBWSymB21}(\eta\otimes\eta)(\Delta(\Phi^\Pi))=(\eta\otimes\eta)(\Phi^\Pi\otimes \epsilon+\epsilon\otimes \Phi^\Pi)=
m_\Pi\otimes \epsilon+\epsilon\otimes m_\Pi
=\Delta(m_\Pi)\end{equation} for each nonsplitable partition $\Pi$. So since $\eta\otimes\eta:\BWSym\otimes \BWSym\longrightarrow \mathcal B^2_1\otimes \mathcal B^2_1$ is a morphism of algebras, the equality $(\eta\otimes\eta)(\Delta(\Phi^{\Pi_1}\cdots\Phi^{\Pi_k}))= \Delta(m_{\Pi_1}\star\cdots\star m_{\Pi_k})$ holds for any $k$-tuples $(\Pi_1,\dots,\Pi_k)$ of  nonsplitable set partitions into lists. This proves that $\eta$ is a morphism of bigebra and implies our statement.
\end{proof}

\section{Conclusion}
In this paper, we have described a combinatorial Hopf algebra which gives a diagrammatic representation of the calculations involved in the normal boson ordering.
This algebra is rather closed to the one proposed by Blasiak \emph{et al.}, but there are many differences which can be exploited to understand soundly the combinatorial aspects of these computations. First, the underlying objects are slightly different. Our objects are graphs with labeled vertices whilst those of Blasiak \emph{et al.} are unlabeled; furthermore our graph have vertices which have the same number of inner edges and outer edges and each edge is either valid or stump. So our algebra is bigger.   
The algebra of Blasiak \emph{et al.} specializes to the enveloping algebra of the Heisenberg Lie algebra which is  described in terms of quotient as follows: $\mathcal U\left(\mathcal L_{\mathcal H}\right)\sim \mathbb C\langle \aa^\dag,\aa,\ee'\rangle/_{\mathcal J'}$ where $\mathcal J'$ is the ideal generated by the three polynomials $[\aa^\dag,\aa]-\ee'$, $[\aa^\dag,\ee']$, and $[\aa,\ee']$. The role of the letter $\ee'$ consists in collecting the statistic of the number of edges denoted by $|\Gamma_0|$ in \cite{BPDSHP} and by $\tau(G)$ in our paper. Consider the algebra $\mathcal H'$ obtained by adding a central element $\ee$ to $\mathcal U\left(\mathcal L_{\mathcal H}\right)$. This algebra allows us to take into account both the statistics $h_c(G)$ and $\tau(G)$. Indeed, it suffices to consider the morphism of algebra $\mathfrak p'_{\mathcal H}$ sending each $\raa_{\rangle\left(i_1\atop j_1\right)\cdots\left(i_k\atop j_k\right)}$ to $\aa^\dag\ee'^{k-1}$, each element of $\mathcal A_{\times\langle}$ to $\aa$, and each element of $\mathcal A_{\langle\ \rangle}$ to $\ee$.
Remarking that $\sum_{p\in\mathrm{Paths}(G)}\left(\ell(\mathrm{seq}(p))-1\right)=\tau(G)$ where $\ell(s)$ denotes the length of the sequence $s$, we obtain
\begin{equation}
\mathfrak p'_{\mathcal H}(\mathtt w(G))=\left(\aa^\dag\right)^{h_f^\downarrow(G)}\aa^{h_f^\uparrow(G)}\ee^{h_c(G)}\ee'^{\tau(G)}.
\end{equation}
Now, the multiplication formula  reads
\begin{equation}\label{multform2}
\left(\aa^\dag\right)^m\aa^n\ee^q\ee'^{v}.\left(\aa^\dag\right)^r\aa^s\ee^t\ee'^{w}=\sum_{i=0}^{\min\{n,r\}}i!\binom{q}i\binom ri \left(\aa^\dag\right)^{m+r-i}\aa^{n+s-i}\ee^{q+t}\ee'^{v+w+i}.
\end{equation}

Section \ref{subalgebras} gave two examples of subalgebras which are related to combinatorial objects. These construction can be generalized to a family of Hopf combinatorial algebras generated by colored partitions \cite{ABCLM}. 
Some combinatorial properties of these objects can be deduced from simple manipulations. For instance, the number of set partitions into lists of $n$ is equal to the number of set partitions of $2n$ such that each part contains at most one even number and, in that case, this number is the minimum of the part. This comes from the morphism $\mathcal B^2_1\longrightarrow \mathcal B^1_1$ sending the element
 \begin{center}\begin{tikzpicture}
\setcounter{Edge}{1}
\setcounter{Vertex}{1}
\bugdxuu 00{b1}
\end{tikzpicture}
\end{center}
 to
 \begin{center}
 \begin{tikzpicture}
\setcounter{Edge}{1}
\setcounter{Vertex}{1}
\bugdu {1.5}0{b1}
\bugdu {2}0{b2}
\end{tikzpicture}
\end{center}
together with the interpretations in terms of set partitions and set partitions into lists described in section \ref{subalgebras}. For instance, the explicit isomorphism sends $\{[4,5,3,1],[2]\}$ to $\{\{1,5,7\},\{2\},\{3\},\{4\}, \{6\},\{8,9\},\{10\}\}$. This comes from the correspondence
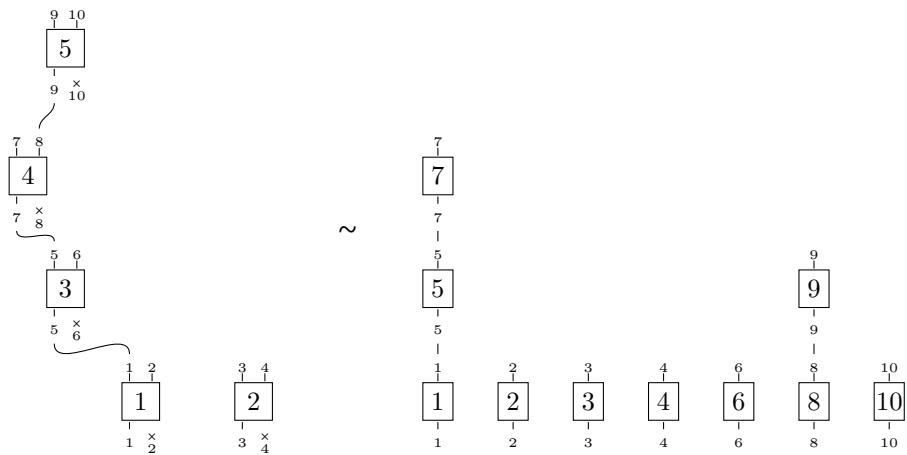
\begin{center}
 \begin{tikzpicture}
\setcounter{Edge}{1}
\setcounter{Vertex}{1} 
\bugdxuu {1}{-5}{b1}
\bugdxuu {2.5}{-5}{b2}
\bugdxuu {0}{-3.5}{b3}
\bugdxuu {-0.5}{-2}{b4}
\bugdxuu {0}{-0.3}{b5}
\draw (d1b3) edge[in=90,out=270] (u1b1);
\draw (d1b4) edge[in=90,out=270] (u1b3);
\draw (d1b5) edge[in=90,out=270] (u2b4);
\node (p1) at (4,-2.5) {\Large$\sim$};
\setcounter{Edge}{1}
\setcounter{Vertex}{1} 

\bugdu {5}{-5}{b1}
\bugdu {6}{-5}{b2}
\bugdu {7}{-5}{b3}
\bugdu {8}{-5}{b4}
\bugdu {5}{-3.5}{b5}
\bugdu {9}{-5}{b6}
\bugdu {5}{-2}{b7}
\bugdu {10}{-5}{b8}
\bugdu {10}{-3.5}{b9}
\bugdu {11}{-5}{b10}
\draw (d1b5) edge[in=90,out=270] (u1b1);
\draw (d1b7) edge[in=90,out=270] (u1b5);
\draw (d1b9) edge[in=90,out=270] (u1b8);

\end{tikzpicture}
\end{center}
Applying this strategy to various subalgebras, we find interpretations of some generalizations of Bell polynomials, like $r$-Bell \cite{mezo} or $(r\sb 1,\dots,r\sb p)$-Bell \cite{mih12}, in terms of $B$-diagrams. 
All these investigations are relegated to a forthcoming paper.

\end{document}